\documentclass[11pt,a4paper]{article}

\usepackage{amsmath,amssymb,amscd,amsthm}
\usepackage{graphics}
\usepackage{enumerate}
\input xy
\xyoption{all}

\usepackage{comment}
\usepackage{tikz}

\usetikzlibrary{arrows}

\usepackage[paperwidth=210mm,
            paperheight=297mm,
            a4paper,
            twoside,
            top=1in, left=1.2in, right=1.2in, bottom=1in]{geometry}

\usepackage[active]{srcltx}

\newtheorem{lemma}{Lemma}

\newtheorem{prop}[lemma]{Proposition}
\newtheorem{cor}[lemma]{Corollary}

\newtheorem{defi}[lemma]{Definition}
\newtheorem{thm}[lemma]{Theorem}
\newtheorem{ex}[lemma]{Example}
\newtheorem{rmk}[lemma]{Remark}

\newcommand{\oln}{{\bf n}}
\newcommand{\olh}{\overline{h}}

\newcommand{\uh}{\mathcal{U}(\mathfrak{h})}
\newcommand{\ug}{\mathcal{U}(\mathfrak{sl}_{n+1})}
\newcommand{\g}{\mathfrak{sl}_{n+1}}
\newcommand{\poln}{\mathbb{C}[h_{1}, \ldots , h_{n}]}

\newcommand{\bb}[1]{\mathbb{#1}}
\newcommand{\fr}[1]{\mathfrak{#1}}
\newcommand{\ca}[1]{\mathcal{#1}}
\newcommand{\pol}[1]{\mathbb{C}[h_{1}, \ldots , #1 , \ldots , h_{n}]}

\usepackage{float}
\restylefloat{figure}

\begin{document}

\title{Simple $\mathfrak{sl}_{n+1}$-module structures on $\uh$}
\author{Jonathan Nilsson}
\date{}
\maketitle

\begin{abstract}
\noindent We study the category $\ca{M}$ consisting of $\ug$-modules whose restriction to 
$\uh$ is free of rank $1$, in particular we classify isomorphism classes of objects in $\ca{M}$ and determine their submodule structure.
This leads to new $\mathfrak{sl}_{n+1}$-modules.
For $n=1$ we also find the central characters and 
derive an explicit formula for taking tensor product with a simple finite dimensional module.
\end{abstract}

\section{Introduction and description of the results}

Classification of simple modules is an important first step in understanding the representation theory
of an algebra. The Lie algebra $\mathfrak{sl}_2$ is the only simple finite dimensional complex Lie algebra
for which some version of such a classification exists, see \cite{Bl,Ma}. For other Lie algebras no full
classification is known, however, various natural classes of simple modules are classified, in particular,
simple finite dimensional modules (see e.g. \cite{Ca,Dix}), simple highest weight modules (see e.g. \cite{Dix,Hum2}),
simple weight modules with finite dimensional weight spaces (see \cite{Fe,Fu,Mathieu}),
Whittaker modules (see \cite{Kos}), and Gelfand-Zetlin modules (see \cite{DFO}).

This paper contributes with a new family of simple modules for the Lie algebra 
$\mathfrak{sl}_{n+1}=\mathfrak{sl}_{n+1}(\mathbb{C})$,
the algebra of $(n+1) \times (n+1)$ complex matrices with trace zero with the Lie bracket $[a,b]=ab-ba$. This algebra
is important as, by Ado's Theorem (see e.g. \cite{Dix}), every finite dimensional complex Lie algebra 
is isomorphic to a subalgebra of $\mathfrak{sl}_{n+1}$ for some $n$. Furthermore,
the representation theory of $\mathfrak{sl}_{n+1}$ has (recently discovered) connection to knot invariants, 
see e.g. \cite{Stroppel,Maz2}. The idea of our construction originates in the attempt to understand whether
the general setup for study of Whittaker modules proposed in \cite{BM} can be used to construct some 
explicit families of simple $\mathfrak{sl}_{n+1}$-modules (in analogy as, for example, was done for the 
Virasoro algebra in \cite{MW,MZ,LLZ}).

Let $\mathfrak{h}$ denote the standard Cartan subalgebra of $\mathfrak{sl}_{n+1}$ consisting of all diagonal matrices.
One of the most classical families of $\mathfrak{sl}_{n+1}$-modules is the family of so-called {\em weight modules}
which are the modules on which $\mathfrak{h}$ acts diagonalizably. In the present paper we study the category $\ca{M}$ 
of $\mathfrak{sl}_{n+1}$-modules  defined by the ``opposite condition'', namely, as the full subcategory 
of $\mathfrak{sl}_{n+1}$-mod consisting of modules which are free of rank $1$ when restricted to $\uh$.
Here is a classical $\mathfrak{sl}_{2}$-example (see \cite{AP}):

\begin{ex}
 Let $n=1$ and let $h=\frac{1}{2}(e_{1,1}-e_{2,2})$. Then $\bb{C}[h]$ becomes an $\mathfrak{sl}_{2}$-module with the action given by
\begin{align*}
 h \cdot f(h) &= hf(h)\\
 e_{1,2} \cdot f(h) &= hf(h-1)\\
 e_{2,1} \cdot f(h) &= -hf(h+1).
\end{align*}
We note that $\mathrm{Res}_{\uh}^{\mathcal{U}(\mathfrak{sl}_{2})} \bb{C}[h]$ is isomorphic to $_{\uh}\uh$ (the module $\uh$ with the natural left action) so under this action,
 $\bb{C}[h]$ is free of rank $1$ and this module belongs to $\ca{M}$.
\end{ex}
Since $\uh$ acts freely on modules in $\ca{M}$, these modules are infinite dimensional and has no weight vectors.
In particular $\ca{M}$ contains no weight modules. We shall later also see that the action of the subalgebra
$\mathfrak{n}_{+}$ of upper triangular matrices is generically not locally finite, so the modules in $\ca{M}$ are generically not Whittaker modules in the sense of \cite{Kos}, or quotients of such.
The modules of $\ca{M}$ are generically not even Whittaker modules in the sense of \cite{BM}.

Classifying the objects of $\ca{M}$ is equivalent to finding all possible ways of extending $_{\uh}\uh$ to an $\mathfrak{sl}_{n+1}$-module. 
In Section~\ref{sl2}  we focus on the case $n=1$. In Theorem~\ref{sl2classi} we classify the modules of $\ca{M}$ for $\mathfrak{sl}_{2}$ and determine their Jordan-H\"{o}lder composition.
 It turns out that the situation is analogous to that of Verma modules (see \cite{Dix,Hum2}) 
 in the following sense: the modules of $\ca{M}$ are generically simple, and each reducible module of $\ca{M}$
 has a unique submodule which also belongs to $\ca{M}$, and a corresponding finite dimensional quotient. In Section~\ref{tensoring} we solve a version of the Clebsch-Gordan problem in $\ca{M}$ for $n=1$:
 we give an explicit decomposition formula for $M \otimes E$ where $M$ is a simple object of $\ca{M}$ and $E$ is a simple finite dimensional module.

In Section~\ref{sln} we generalize some of the results to $\mathfrak{sl}_{n+1}$ for arbitrary $n\geq 1$. In particular, we classify isomorphism classes of objects in $\ca{M}$ completely in Theorem~\ref{class}.
 Here follows a special case of the result:
\begin{thm}
For $1 \leq k \leq n$, let $h_{k}:=e_{k,k}-\frac{1}{n+1}\sum_{m=1}^{n+1}e_{m,m} \in \mathfrak{sl}_{n+1}$. Then for each $b \in \bb{C}$, the vector space $\bb{C}[h_{1}, \ldots, h_{n}]$
 is a simple $\mathfrak{sl}_{n+1}$-module under the action
 \[ \begin{cases}
 h_{i} \cdot f(h_{1}, \ldots, h_{n}) = h_{i}f(h_{1}, \ldots, h_{n}) & 1 \leq i \leq n,  \\ 
 e_{i,n+1} \cdot f(h_{1}, \ldots, h_{n}) = (b+\sum_{m=1}^{n}h_{m})(h_{i}-b-1)f(h_{1} \ldots,h_{i}-1, \ldots, h_{n}), &  1 \leq i \leq n,  \\ 
 e_{n+1,j} \cdot f(h_{1}, \ldots, h_{n}) =-f(h_{1}, \ldots,h_{j}+1, \ldots, h_{n}) &   1 \leq j \leq n, \\ 
 e_{i,j} \cdot f(h_{1}, \ldots, h_{n}) =(h_{i}-b-1)f(h_{1}, \ldots ,h_{i}-1, \ldots, h_{j}+1, \ldots, h_{n}) &  1 \leq i,j \leq n.  \\ 
\end{cases}
\]
Moreover, for $n > 1$, different choices of $b$ gives nonisomorphic modules.
For details, compare this with Definition~\ref{structure} where the module structure above corresponds to the module labeled $M_{b}^{\varnothing}$.
\end{thm}

In Section~\ref{simplicity} we determine the submodule structure of the objects in $\ca{M}$. It turns out that the objects generically are simple, while the reducible ones
have length $2$ with a simple finite dimensional top.\\

After this paper was finished we heard about some related results~\cite{TZ} which are to appear shortly.\\

\noindent
{\bf Acknowledgements}
I am very grateful to Volodymyr Mazorchuk for his helpful advice while I was writing this paper.
I also want to thank Professor Kaiming Zhao for informing us about~\cite{TZ}.

\section{Preliminaries}
\label{section2}
In this section we collect some of the basic definitions and results needed for studying our module categories.
We denote by $\mathbb{N}$ and $\mathbb{N}_0$ the sets of positive integers and nonnegative integers, respectively.

\subsection{Some categories of modules}
In this subsection let $\mathfrak{g}$ be any finite dimensional complex Lie algebra admitting a triangular decomposition
\begin{displaymath}
 \mathfrak{g} = \mathfrak{n_{-}} \oplus \mathfrak{h} \oplus \mathfrak{n_{+}}.
\end{displaymath}
As usual we let $\mathcal{U}(\mathfrak{g})\text{-Mod}$ be the category of all left $\mathfrak{g}$ modules while we denote the subcategory of finite dimensional modules by $\mathfrak{g}\text{-fmod}$.
We also let $\ca{O}$ be the full subcategory of  $\mathcal{U}(\mathfrak{g})\text{-Mod}$ consisting of finitely generated weight modules which are locally $\mathcal{U}(\mathfrak{n_{+}})$-finite~\cite{BGG,Hum2}.
Now define $\mathfrak{M}$ to be the full subcategory of $\mathfrak{g}\text{-Mod}$ consisting of modules whose restriction to $\uh$ is finitely generated.
\begin{prop}
\begin{enumerate}[$($i$)$]
 \item\label{cat1}  The category $\mathfrak{M}$ is abelian.
 \item\label{cat2}  The only weight modules in $\mathfrak{M}$ are the finite dimensional modules. In particular $\mathfrak{M} \cap \ca{O} = \mathfrak{g}\text{-fmod}.$
\end{enumerate}
\end{prop}
\begin{proof}
To prove \eqref{cat1}, first note that the category $\mathfrak{M}$ is closed under taking quotients and direct sums.
Moreover, $\uh$ is isomorphic to the polynomial algebra in finitely many variables so it is noetherian.
Thus any submodule of a finitely generated module is finitely generated and $\mathfrak{M}$ is closed under taking submodules. It follows that $\mathfrak{M}$ is abelian.

To prove \eqref{cat2}, first note that any finite-dimensional weight module is generated as an $\uh$-module 
by any basis, so $\mathcal{U}(\mathfrak{g})\text{-fmod} \subset \mathfrak{M}$.
On the other hand, if $M$ is a weight module and $\{v_{i}| i \in I\}$ is a basis of $M$ consisting of weight vectors, then as an $\uh$-module, $M$ decomposes as a direct sum of one dimensional modules:
\[M=\bigoplus_{i\in I} \bb{C}v_{i}.\]
In particular, if $I$ is infinite, no finite subset can generate $M$ as a $\uh$-module.
\end{proof}
Now define $\ca{M}$ to be the full subcategory of $\mathcal{U}(\mathfrak{g})\text{-Mod}$ consisting of objects whose restriction to $\uh$ are free of rank $1$:
\[\ca{M}:= \{M \in \mathcal{U}(\mathfrak{g})\text{-Mod} | Res_{\uh}^{\mathcal{U}(\mathfrak{g})} M \simeq _{\uh}{\uh}\}.\] The goal of this paper is
 to understand the category $\ca{M}$ for $\mathfrak{g}=\mathfrak{sl}_{n+1}$.
We also define \[\overline{\ca{M}}:= \{M \in \mathcal{U}(\mathfrak{g})\text{-Mod} | Res_{\uh}^{\mathcal{U}(\mathfrak{g})}M \text{ is free of finite rank}\}.\]
Then $\overline{\ca{M}}$ is closed under finite direct sums and under taking tensor products with finite dimensional modules.
We now have inclusions of full subcategories as follows:
\[\ca{M} \subset \overline{\ca{M}} \subset \mathfrak{M} \subset \mathcal{U}(\mathfrak{g})\text{-Mod}.\]
For $M \in \overline{\ca{M}}$ we note that $\uh$ acts freely on $M$. Thus the sum of the weight spaces of $M$ is zero and, in particular, $\mathcal{U}(\mathfrak{g})\text{-fmod} \cap \overline{\ca{M}}=\varnothing$.

\subsection{A basis of $\uh$}
For the rest of this paper, we fix $\mathfrak{g}=\mathfrak{sl}_{n+1}$.
For each $k \in \oln :=\{1,2,\ldots, n\}$ define
\begin{equation}
\label{hdef}
 h_{k}:=e_{k,k}-\frac{1}{n+1}\sum_{i=1}^{n+1}e_{i,i}.
\end{equation}
Then $\{h_{1}, \ldots , h_{n}\}$ generate $\mathcal{U}(\mathfrak{h})$ and we can identify $\mathcal{U}(\mathfrak{h}) \simeq \mathbb{C}[h_{1}, \ldots , h_{n}]$.
An advantage of this choice of generators is that they satisfy the relations
\[ [h_{k},e_{i,n+1}] = \delta_{k,i}e_{i,n+1} \quad \text{ for all } i,k \in \oln.\]
We also define \[\olh := \sum_{i=1}^{n} h_{i}\] and note that $[\olh, e_{i,n+1}] = e_{i,n+1}$ for all $i\in \oln$.
With respect to the basis 
$$\{e_{i,j} | 1 \leq i,j \leq n+1; i \neq j\} \cup \{h_{1}, \ldots, h_{n}\}$$ of $\g$, the bracket
 operation is given by the following lemma.
\begin{lemma}
\label{lemma1}
For $i,i',j,j' \in {\bf n+1}$; $i \neq j$; $i'\neq j'$; and $k,k' \in \oln$ we have
\begin{align*}
[e_{i,j},e_{i',j'}] & = \delta_{j,i'}e_{i,j'} - \delta_{i,j'}e_{i',j}\\
[h_{k},e_{i,j}]&=(\delta_{k,i}-\delta_{k,j})e_{i,j}\\
[h_{k},h_{k'}]&=0.
\end{align*}
\end{lemma}
\begin{proof}
The first and third identities are obvious and the second identity follows from \eqref{hdef}.
\end{proof}

We also introduce some notation for our polynomial rings. Define
\[\ca{P}:=\poln,\] and for each $i \in \oln$ define
\[\ca{P}_{i}:=\pol{h_{i-1},h_{i+1}}.\]
Note that $\ca{P} \simeq \ca{P}_{i}[h_{i}] \simeq (\ca{P}_{i} \cap \ca{P}_{j})[h_{i},h_{j}]$ and so on. 

\subsection{Gradings}
It turns out to be helpful to use some different gradings on $\ca{P}$.
 For each $i \in \oln$ we define
\[\deg_{i} h_{1}^{d_{1}}h_{2}^{d_{2}} \cdots h_{n}^{d_{n}}:=d_{i}.\] In other words, $deg_{i} (f)$ is the degree of $f$ when considered
 as a polynomial in a single variable $h_{i}$ and with coefficients in $\ca{P}_{i}$. 
For convenience we let $\deg_{i} 0 := -1$. We also define \[\mathtt{c}_{i}:\ca{P} \rightarrow \ca{P}_{i},\] to be the map taking the leading
 coefficient of a given polynomial with respect to the grading $\deg_{i}$. For example, for $f=h_{1}h_{2}+h_{2} \in \bb{C}[h_{1},h_{2},h_{3}]$ we have
\[\mathtt{c}_{1}(f) = h_{2}, \quad \mathtt{c}_{2}(f) = h_{1}+1 \quad \text{ and }  \quad\mathtt{c}_{3}(f)=f.\]
Note that each map $\mathtt{c}_{i}$ is nonlinear but multiplicative:
\[\mathtt{c}_{i}(fg) = \mathtt{c}_{i}(f)\mathtt{c}_{i}(g) \quad \text{ for all }\quad  f,g\in\ca{P}.\]

\subsection{Automorphisms of $\uh$}
\label{sigmas}
For each $i \in \oln$, define an algebra automorphism \[ \sigma_{i}: \uh \longrightarrow \uh\]
by \[\sigma_{i}(h_{k}):=h_{k}-\delta_{i,k}.\]
Then $\sigma_{1}, \ldots , \sigma_{n}$ generate an abelian subgroup of $Aut(\uh)$. For $f\in \ca{P}$ we explicitly have
\[\sigma_{i}(f(h_{1}, \ldots, h_{n})) = f(h_{1}, \ldots, h_{i}-1, \ldots , h_{n}).\]
Note also that $\mathtt{c}_{i}(\sigma_{i} f) = \mathtt{c}_{i}(f)$.

We shall later have to solve some equations involving these automorphisms, so we collect some basic facts about them here.
\begin{lemma}
\label{deg}
 Let $f \in \ca{P}$ be a nonzero polynomial. Then
 \[\deg_{i} (\sigma_{i}(f)-f) = (\deg_{i}f)-1.\]
\end{lemma}
\begin{proof}
Write $f=\sum_{k=0}^{m} h_{i}^{k} f_{k}$ with $f_{k} \in \ca{P}_{i}$ and $f_{m} \neq 0$. Then 
\[\sigma_{i}(f)-f = \sum_{k=0}^{m} ((h_{i}-1)^{k}-h_{i}^{k}) f_{k}, \] and we see that the coefficient at $h_{i}^{m}$ is $0$ while  the coefficient at $h_{i}^{m-1}$ is precisely $mf_{m}$ which is nonzero.
\end{proof}
\begin{lemma}
\label{lemmasigmaeq}
 For every $g \in \ca{P}$ the equation 
\begin{equation}
\label{sieq}
 \sigma_{i}(f)-f=g
\end{equation}
has a solution $f$ and this solution is unique up to addition of elements from $\ca{P}_{i}$.
\end{lemma}
\begin{proof}
 We note that the left side of \eqref{sieq} is linear in $f$. Thus the general solution is the sum of a particular solution and an arbitrary solution to the corresponding homogeneous equation
\begin{equation}
\label{homeq}
 \sigma_{i}(f)-f=0.
\end{equation}
Applying $\deg_{i}$ to \eqref{homeq}, we obtain $(\deg_{i}f) - 1=-1$. This means that any solution to \eqref{homeq} 
has the form  $f \in \ca{P}_{i}$.
Next we claim that $\sigma_{i}(f)-f=h_{i}^{k}$ has a solution $f_{k}$ for each $k \in \bb{N}_{0}$. This is true by induction: for $k=0$ a solution is $f_{0}=-h_{i}$, and supposing it has a solution for all $p<k$
we note that $\sigma_{i}(-\frac{h_{i}^{k+1}}{k+1})-(-\frac{h_{i}^{k+1}}{k+1}) = h_{i}^{k} +g$ for some $g \in \bb{C}[h_{i}]$ with $\deg g < k$. Thus
\[f=-\frac{h_{i}^{k+1}}{k+1} - \tilde{f}\]
is a solution to $\sigma_{i}(f)-f=h_{i}^{k}$, where $\tilde{f}$ is any solution to 
$\sigma_{i}(\tilde{f})-\tilde{f}=g$. This proves that $\sigma_{i}(f)-f=h_{i}^{k}$ has a solution $f_{k}$ for every $k\in\bb{N}_{0}$. 
Now, finally, we note that for $f \in \ca{P}_{i}$ we have $\sigma_{i}(fp)-fp=f(\sigma_{i}(p)-p)$, so, writing \[g=\sum_{k=0}^{m}g_{k}h_{i}^{k}\] for some
$g_{k} \in \ca{P}_{i}$, a solution to $\sigma_{i}(f)-f=g$ is \[f=\sum_{k=0}^{m}g_{k}f_{k}.\]
\end{proof}

\subsection{Automorphisms of $\ug$}
\label{functors}
Denote by $\tau$ the involutive automorphism of $\ug$ defined by $\tau: e_{i,j} \mapsto -e_{j,i}$. It is easy to check that, when restricted to $\uh$, the automorphism $\tau$ satisfies
\[\tau \sigma_{i} = \sigma_{i}^{-1} \tau \text{ for all } i \in \oln.\] 
Note also that on $\uh$ we explicitly have $\tau(f(h_{1}, \ldots, h_{n})) = f(-h_{1}, \ldots, -h_{n})$.

For each $a=(a_{1}, \ldots , a_{n+1}) \in (\bb{C}^{*})^{n+1}$, we also define an automorphism \[\varphi_{a}: \ug \longrightarrow \ug\]
by \[\varphi_{a}: e_{i,j} \mapsto \frac{a_{i}}{a_{j}} e_{i,j}.\]
This is well defined since 
\begin{align*}
 \varphi_{a}(e_{i,j}e_{k,l}-e_{k,l}e_{i,j}) &= \frac{a_{i}a_{k}}{a_{j}a_{l}}(e_{i,j}e_{k,l}-e_{k,l}e_{i,j})
 = \frac{a_{i}a_{k}}{a_{j}a_{l}}([e_{i,j},e_{k,l}])\\
 &= \frac{a_{i}a_{k}}{a_{j}a_{l}}(\delta_{k,j}e_{i,l}- \delta_{i,l}e_{k,j})
 = \delta_{k,j}\frac{a_{i}}{a_{l}}e_{i,l}-  \delta_{i,l}\frac{a_{k}}{a_{j}}e_{k,j}\\
 &= \varphi_{a}(\delta_{k,j}e_{i,l}-  \delta_{i,l}e_{k,j})
 = \varphi_{a}([e_{i,j},e_{k,l}]).
\end{align*}
Note that \[ \tau \circ \varphi_{a} = \varphi_{a^{-1}} \circ \tau,\]
where the inverse is taken componentwise.

Each automorphism $\varphi \in \ug$ induces a functor
\[\mathrm{F}_{\varphi}: \g\text{-Mod} \longrightarrow \g\text{-Mod}\] which maps each module to itself (as a set) but with a new action $\bullet$ defined by
\[x\bullet v:=\varphi(x)\cdot v.\] The functor $\mathrm{F}_{\varphi}$ maps morphisms to themselves.

We will write $\mathrm{F}_{a}:=\mathrm{F}_{\varphi_{a}}$ and $\overline{M}:=F_{\tau}(M)$.
Note that $\tau: \overline{\uh} \rightarrow \uh$ is an isomorphism of left $\uh$-modules.
This follows from the fact that for $f,g \in \uh$ we have $\tau(f \bullet g) = \tau(\tau(f)\cdot g)= \tau(\tau(f) g) = f\cdot \tau (g)$.

We collect some of the properties of the above functors in a lemma. Multiplication and inversion in $(\bb{C}^{*})^{n+1}$ are defined pointwise.
\begin{lemma}
\label{functorlemma}
The functors $\mathrm{F}_{a}$ and $\mathrm{F}_{\tau}$ have the following properties:
 \begin{enumerate}[$($i$)$]
  \item\label{fun1} $\mathrm{F}_{(1,1, \ldots, 1)} = \mathrm{Id}_{\g\text{-}\mathrm{Mod}}$.
  \item\label{fun2} $\mathrm{F}_{a} = \mathrm{F}_{\lambda a}$ for all $\lambda \in \bb{C}^{*}$.
  \item\label{fun3} $\mathrm{F}_{a} \circ \mathrm{F}_{b} = \mathrm{F}_{a \cdot b}$.
  \item\label{fun4} $\mathrm{F}_{a}^{-1} \simeq \mathrm{F}_{a^{-1}}$.
  \item\label{fun5} For any $M \in \g\text{-}\mathrm{fmod}$ we have $\mathrm{F}_{a}(M) \simeq M \simeq \mathrm{F}_{\tau}(M)$.
  \item\label{fun6} $\mathrm{F}_{a}\ca{M}=\ca{M} = \mathrm{F}_{\tau} \ca{M}$.
  \item\label{fun7} For any $M \in \ca{M}$, $\mathrm{F}_{a}(M) \simeq M$ only if $a\in \bb{C}^{*}(1,1,\ldots, 1)$.
  \item\label{fun9} $\mathrm{F}_{\tau} \simeq \mathrm{F}_{\tau^{-1}}$.
  \item\label{fun10} $\mathrm{F}_{\tau} \circ  \mathrm{F}_{a} \simeq \mathrm{F}_{a^{-1}} \circ \mathrm{F}_{\tau}$.
  \item\label{fun8} $\mathrm{F}_{a}$ and $\mathrm{F}_{\tau}$ are auto-equivalences.
 \end{enumerate}
\end{lemma}
\begin{proof}
Claims \eqref{fun1} - \eqref{fun4} follow directly from the definition of $\mathrm{F}_{a}$. 
 To prove claim \eqref{fun5} we note that if
 $v$ is a weight vector of weight $\lambda$ in $M$, then for all $h \in \fr{h}$, in $\mathrm{F}_{a}(M)$ we have
\[h\bullet v=\varphi_{a}(h)\cdot v = \lambda(h)v,\] which shows that $v$ is still a weight vector of weight $\lambda$ in $\mathrm{F}_{a}(M)$.
Similarly, in $\mathrm{F}_{\tau}(M)$ we have
\[h\bullet v=\tau(h)\cdot v = -h \cdot v = -\lambda(h)v,\]
so $v$ is still a weight vector, but now with weight $-\lambda$.
But finite dimensional modules are uniquely determined up to isomorphism by their characters (that is, their occurring weights and the dimension of the corresponding weight spaces),
so we have $\mathrm{F}_{a}(M) \simeq M$. Similarly, we know that the dimension of the weight space of $\lambda$ and $-\lambda$ are equal in any finite dimensional module,
so we also obtain $\mathrm{F}_{\tau}(M) \simeq M$.
To prove claim  \eqref{fun6}, let $M \in \ca{M}$. We first note that $\mathrm{F}_{a}(M)$ is still equal to $M$ as a set, and the action of $\fr{h}$ is the same since $\varphi$ fixes $\uh$ pointwise 
so $\mathrm{F}_{a}(M)$ is still free of rank $1$ over $\uh$. Similarly, as we noted before, $\tau: \overline{\uh} \rightarrow \uh$ is an isomorphism so $\mathrm{F}_{\tau}(M)$ is still free of rank $1$ over $\uh$. 
 To prove claim \eqref{fun7},
suppose $\Phi:\mathrm{F}_{a}(M) \rightarrow M$ is an isomorphism where $a \not\in \bb{C}^{*}(1,1,\ldots, 1)$. Since $\Phi(f)=f\bullet \Phi(1)=f \cdot \Phi(1)$, $\Phi$
 is determined by $\Phi(1)$ and the same is true for
 $\Phi^{-1}$. Since $1=\Phi^{-1}(\Phi(1)) = \Phi^{-1}(1)\Phi(1)$, we obtain $\Phi(1)=c \in \mathbb{C}^{*}$ and thus $\Phi(f)=cf$. Pick indices $i,j$ such that $a_{i} \neq a_{j}$. 
Then
\begin{equation}
\label{eqphi1}
c(e_{i,j} \cdot 1) =e_{i,j} \cdot \Phi(1) = \Phi(e_{i,j}\bullet 1)=
\Phi(\frac{a_{i}}{a_{j}}e_{i,j} \cdot 1)=c\frac{a_{i}}{a_{j}}(e_{i,j} \cdot 1).
\end{equation}
 However,  $(e_{i,j} \cdot 1)$ must be nonzero, as otherwise
 $[e_{i,j},e_{j,i}] \cdot 1=0$ which is impossible since $[e_{i,j},e_{j,i}] \in \fr{h}$. Thus \eqref{eqphi1} does not hold, which shows that there exists no such isomorphism $\Phi$.
  Claim \eqref{fun9} and \eqref{fun9} are obvious from the corresponding relations in $\ug$ and claim \eqref{fun10} is a straightforward calculation.
Finally,  claim \eqref{fun8} is clear since $\mathrm{F}_{a} \circ \mathrm{F}_{a^{-1}} \simeq \mathrm{Id}_{\g-\text{Mod}} \simeq \mathrm{F}_{a^{-1}} \circ \mathrm{F}_{a}$
 and $\mathrm{F}_{\tau} \circ \mathrm{F}_{\tau} \simeq \mathrm{Id}_{\g-\text{Mod}}$.
\end{proof}

\subsection{Action of Chevalley generators}
Let $M \in \ca{M}$. Since $M$ is free of rank $1$ we have an isomorphism $\varphi: M \rightarrow \uh$ in $\uh$-Mod. By defining $x \bullet f :=\varphi( x \cdot \varphi^{-1}(f))$
 for all $x \in \ug$, $f \in \uh$, the space $\uh$ becomes a $\ug$-module isomorphic to $M$ via the map $\varphi$. Thus, to classify the isomorphism classes of objects in $\ca{M}$,
 we need only consider all possible extensions of the natural left $\uh$-action on $_{\uh}\uh$ to $\ug$.
\begin{prop} 
\label{prop1}
 Let $M \in \ca{M}$. Then, identifying $M$ as a vector space with $\ca{P}$, the action of $\g$ on $M$ is completely determined by the action of the Chevalley generators on $1 \in M$.
 Explicitly, for $f\in \ca{P}$ we have
\begin{displaymath}
\begin{array}{rcll}
 h_{k} \cdot f &=& h_{k} f &  k \in \oln,\\
 e_{i,n+1} \cdot f &=& p_{i} \sigma_{i} f  &  i \in \oln,\\
 e_{n+1,j} \cdot f &=& q_{i} \sigma_{j}^{-1} f  &  j \in \oln,\\
 e_{i,j} \cdot f &=& \big( p_{i}\sigma_{i}(q_{j})-q_{j}\sigma_{j}^{-1}(p_{i})\big) \sigma_{i}\sigma_{j}^{-1}f  \qquad&  i,j \in \oln, i \neq j,\\
\end{array}
\end{displaymath}
where  $p_{i}:=e_{i,n+1} \cdot 1$ and $q_{j}:=e_{n+1,j} \cdot 1$ for all $i,j \in \oln$.
\end{prop}
\begin{proof}
Since $M \in \ca{M}$, we know that, both as a vector space and as an $\fr{h}$-module, $M$ is isomorphic to $\ca{P}$. In other words, the action of $\fr{h}$ on $M$ can be written explicitly as 
$h_{k} \cdot f = h_{k}f$ for all $ f \in \ca{M}, \;k \in \oln$.

Now for each $i \in \oln$ we define 
\begin{displaymath}
 p_{i}:=e_{i,n+1} \cdot 1 \in \ca{P}, \qquad\qquad
 q_{i}:=e_{n+1,i} \cdot 1 \in \ca{P}.
\end{displaymath}
Since
$\delta_{k,i}(e_{i,n+1}\cdot f)=[h_{k},e_{i,n+1}]\cdot f = h_{k}(e_{i,n+1}\cdot f)- e_{i,n+1}\cdot (h_{k}f)$, 
we obtain \[e_{i,n+1}\cdot (h_{k}f) = (h_{k}-\delta_{k,i})(e_{i,n+1}\cdot f),\]
which shows that the action of the element 
$e_{i,n+1}$ can be determined inductively from its action on $1$. Explicitly we get
$e_{i,n+1}\cdot f = p_{i}\sigma_{i}f$, using the operator $\sigma_{i}$ from Section~\ref{sigmas}.
Analogous calculations show that $e_{n+1,j}\cdot f = q_{j}\sigma_{j}^{-1}f$.
But then, for all $i,j \in \oln$ with $i\neq j$ we have 
\[e_{i,j}\cdot f = [e_{i,n+1},e_{n+1,j}] \cdot f = e_{i,n+1} \cdot e_{n+1,j} \cdot f - e_{n+1,j} \cdot e_{i,n+1} \cdot f.\]
Explicitly this gives us
\[e_{i,j}\cdot f = \big( p_{i}\sigma_{i}(q_{j})-q_{j}\sigma_{j}^{-1}(p_{i})\big) \sigma_{i}\sigma_{j}^{-1}f.\]
This means that the action of $\g$ on $M$ is completely determined by the $(2n)$-tuple $(p_{1}, \ldots, p_{n},q_{1}, \ldots, q_{n}) \in \ca{P}^{2n}$ as stated in the proposition.
\end{proof}

Note that not every tuple $(p_{1}, \ldots, p_{n},q_{1}, \ldots, q_{n})$ determines an $\g$-module.
We now turn to the converse problem: determine which choices of $(p_{1}, \ldots, p_{n},q_{1}, \ldots, q_{n}) \in \ca{P}^{2n}$
 give rise to a module structure on $\ca{P}$ by the definition of the action as in Proposition~\ref{prop1}.

\begin{prop}
\label{symmprop}
 Suppose that a $2n$-tuple $(p_{1}, \ldots, p_{n},q_{1}, \ldots, q_{n}) \in \ca{P}^{2n}$ gives a $\g$-module via the action in Proposition~\ref{prop1}.
 Then so does the $2n$-tuple
 $$(-\tau(q_{1}), \ldots, -\tau(q_{n}),-\tau(p_{1}), \ldots, -\tau(p_{n})).$$
\end{prop}
\begin{proof}
 Let $M$ be the module defined by $(p_{1}, \ldots, p_{n},q_{1}, \ldots, q_{n})$. Applying the functor $\mathrm{F}_{\tau}$ from Section~\ref{functors} to $M$ we obtain a module $\overline{M}$ which is isomorphic to
 $\overline{\uh}$ and hence to $\uh$ as a left $\uh$ module via the restriction of $\tau$ to $\uh$ (where we identify $\overline{M}$ with $\uh$ as sets).
 Transferring the action of $\ug$ on $\overline{M}$ to $\uh$ via this map $\tau$ we explicitly obtain a $\ug$-module structure on $\uh$ given explicitly by
 \[ e_{i,j} * f =\tau ( e_{i,j} \bullet \tau (f)) = \tau(\tau(e_{i,j}) \cdot \tau(f)).\]
In particular, we obtain $h_{k}*f=h_{k}f$ and
 \[e_{i,n+1} * f = \tau(\tau(e_{i,n+1}) \cdot \tau(f))=\tau(-e_{n+1,i} \cdot \tau(f)) = \tau(-q_{i}\sigma_{i}^{-1}(\tau(f))) =  -\tau(q_{i})\sigma_{i}(f),\]
as well as
 \[e_{n+1,j} * f = \tau(\tau(e_{n+1,j}) \cdot \tau(f))=\tau(-e_{j,n+1} \cdot \tau(f)) = \tau(-p_{i}\sigma_{i}(\tau(f))) =  -\tau(p_{i})\sigma_{i}^{-1}(f),\]
  for all $i,j,k \in \oln$.
 Thus $\overline{M}$ is isomorphic to the module in $\ca{M}$ determined by  \[(-\tau(q_{1}), \ldots, -\tau(q_{n}),-\tau(p_{1}), \ldots, -\tau(p_{n}))\] via the action in Proposition~\ref{prop1}.
\end{proof}

Before considering the problem off classification for arbitrary $n$, we first solve it for $n=1$. 
The solution turns out to be a prototype for the general solution.

\section{The $\mathfrak{sl}_{2}$ case}
\label{sl2}
In this section We consider the case $n=1$. We classify all the objects of $\ca{M}$, determine their submodule structure, find their central character, and derive an explicit formula
for taking tensor product with any simple finite-dimensional module.
\subsection{Classification of objects in $\ca{M}$}
We consider the case $n=1$. As before, we identify modules in $\ca{M}$ (as sets) with $\bb{C}[h]$ where $h:=h_{1}=\frac{1}{2}(e_{1,1}-e_{2,2})$.
We also write $\sigma$ for $\sigma_{1}$.

\begin{defi}
\label{defsl2}
For each $b\in \bb{C}$ define $M_{b}$ to be the vector space $\bb{C}[h]$ equipped with the following $\mathfrak{sl}_{2}$-action:
\begin{displaymath}
\begin{array}{rcl}
 h \cdot f &=& h f  ,\\
 e_{1,2} \cdot f &=&  (h+b)\sigma f ,\\
 e_{2,1} \cdot f &=& -(h-b)\sigma^{-1} f.
\end{array}
\end{displaymath}

Similarly, define $M'_{b}$ to be the vector space $\bb{C}[h]$ equipped with the following $\mathfrak{sl}_{2}$-action:
\begin{displaymath}
\begin{array}{rcl}
 h \cdot f &=& h f  ,\\
 e_{1,2} \cdot f &=&  \sigma f ,\\
 e_{2,1} \cdot f &=& -(h+b+1)(h-b)\sigma^{-1} f.
\end{array}
\end{displaymath}
\end{defi}
\begin{thm}
\label{sl2classi}
\begin{enumerate}[$($i$)$]
\item\label{clasl21} For all $b \in \bb{C}$ the above defines on
$M_{b}$ and $M_{b}'$ the structure of $\mathfrak{sl}_{2}$-modules.
\item\label{clasl22} The set of modules 
\begin{multline*}
\{\mathrm{F}_{(a,1)}(M_{b}) | a \in \bb{C}^{*}, b \in \bb{C}\}\\
 \cup \{\mathrm{F}_{(a,1)}(M_{b}') | a \in \bb{C}^{*}, b \in \bb{C}_{ \geq -\frac{1}{2}}\}\\
 \cup \{\mathrm{F}_{(a,1)} \circ \mathrm{F}_{\tau}(M_{b}') | a \in \bb{C}^{*}, b \in \bb{C}_{ \geq -\frac{1}{2}}\}
\end{multline*}

where $\bb{C}_{ \geq -\frac{1}{2}}=\{z \in \bb{C} | Re(z) \geq -\frac{1}{2}\}$ is a skeleton of $\ca{M}$.
\end{enumerate}
\end{thm}
\begin{proof}
To prove claim \eqref{clasl21} We first check that $M_{b}$ is an $\mathfrak{sl}_{2}$-module. We check the following relations:
\[h \cdot e_{1,2} \cdot f - e_{1,2} \cdot h \cdot f = h((h+b)\sigma f) - (h+b)\sigma (hf) = (h+b)\sigma f = [h,e_{1,2}] \cdot f\]
\[h \cdot e_{2,1} \cdot f - e_{2,1} \cdot h \cdot f = h(-(h-b)\sigma^{-1} f) + (h-b)\sigma^{-1} (hf) = (h-b)\sigma^{-1} f = [h,e_{2,1}] \cdot f\]
\begin{align*}
 e_{1,2} \cdot e_{2,1} \cdot f - e_{2,1} \cdot e_{1,2} \cdot f &= -(h+b)\sigma (h-b)\sigma^{-1}f + (h-b)\sigma^{-1}(h+b) \sigma f\\
 &= \big(-(h+b)(h-b-1) + (h-b)(h+b+1) \big) \sigma \sigma^{-1} f\\
 &= (2h)f =  [e_{1,2} , e_{2,1}] \cdot f.
\end{align*}
The remaining relations are obvious. Thus $M_{b}$ is an $\mathfrak{sl}_{2}$-module. 
It is left to the reader to verify that also $M_{b}'$ is an $\mathfrak{sl}_{2}$-module.

To prove claim \eqref{clasl22}, let $M \in \ca{M}$ and define 
$$p:=p_{1}=e_{1,2} \cdot 1 \in \bb{C}[h]\quad \text{  and }\quad q:=q_{1}=e_{2,1} \cdot 1 \in \bb{C}[h].$$
Then, in accordance with Proposition~\ref{prop1}, the action of $\mathfrak{sl}_{2}$ on $M=\bb{C}[h]$ is given by
\begin{displaymath}
\begin{array}{rcl}
 h \cdot f &=& h f  ,\\
 e_{1,2} \cdot f &=& p \sigma f ,\\
 e_{2,1} \cdot f &=& q \sigma^{-1} f.
\end{array}
\end{displaymath}
Since $M$ is a module, we have, in particular,
 $(2h)f = [e_{1,2},e_{2,1}]\cdot f = e_{1,2}\cdot e_{2,1} \cdot f - e_{2,1}\cdot e_{1,2} \cdot f$
for all $f\in \bb{C}[h]$. This is equivalent to $p \sigma(q)-q\sigma^{-1}p = 2h$.
The latter transforms to $\sigma(g)-g = 2h$ by letting $g=q\sigma^{-1}(p)$.
The equation $\sigma(g)-g = 2h$ has the form discussed in Lemma~\ref{lemmasigmaeq} and the set of solutions is
 $\{-h(h+1)+c \; | c \in \bb{C}\;\}$. Now, every $g$ in this set has a factorization into linear factors of the form
$g=-(h+b+1)(h-b)$ for some $b \in \bb{C}$. Thus we try to find all polynomials $p,q$ for which
$q\sigma^{-1}(p)=-(h+b+1)(h-b)$.
Since $\bb{C}[h]$ is a UFD, we only have three cases, corresponding to $(\deg p, \deg q) \in \{(2,0),(0,2),(1,1)\}$.

If $(\deg p, \deg q) = (0,2)$, then
\[(p,q)=(a,-a^{-1}(h+b+1)(h-b)) \qquad\text{ for some }\qquad  a\in \bb{C}^{*}\]
and hence $M \simeq \mathrm{F}_{(a,1)}(M_{b}')$.
If $(\deg p, \deg q) = (2,0)$, then
\[(p,q)=(a(h+b)(h-b-1),-a^{-1}) \qquad\text{ for some }\qquad a\in \bb{C}^{*}\]
and hence $M \simeq \mathrm{F}_{(a^{-1},1)} \circ \mathrm{F}_{\tau}(M_{b}')$.
Finally, if $(\deg p, \deg q) = (1,1)$, we obtain either
\[(p,q)=(a(h+b),-a^{-1}(h-b)) \qquad\text{ for some }\qquad  a\in \bb{C}^{*},\]
or \[(p,q)=(a(h-b-1),-a^{-1}(h+b+1)) \qquad\text{ for some }\qquad a\in \bb{C}^{*}.\]
In this case  either $M \simeq \mathrm{F}_{(a,1)}(M_{b})$ or  $M \simeq \mathrm{F}_{(a,1)}(M_{-b-1})$.

The above argument shows that any module in $\ca{M}$ is isomorphic to either $\mathrm{F}_{a}(M_{b})$, $\mathrm{F}_{a}(M'_{b})$, or $\mathrm{F}_{a} \circ \mathrm{F}_{\tau}(M'_{b})$
for some $a \in \mathbb{C}^{*}$ and $b \in \bb{C}$.
We now show that these cases are essentially mutually exclusive.  First note that any morphism 
$\Phi$ in $\ca{M}$ is determined by its value at $1$ since $\Phi(f)=f \Phi(1)$.
This also shows that any invertible morphism $\Phi$ must have $\Phi(1)$ invertible, and thus $\Phi$ is 
multiplication by a nonzero constant: $\Phi(f)=cf$.
Now, let $\Phi: M \longrightarrow M'$ be an isomorphism in $\ca{M}$. Let $p:=e_{1,2} \cdot 1 \in M$ and
$p':=e_{1,2} \cdot 1 \in M'$. Then 
\[cp'=p'\sigma \Phi(1) = e_{1,2} \Phi(1) = \Phi(e_{1,2} \cdot 1) = \Phi(p) = p \Phi(1)=cp\]
which gives $p=p'$. Thus two modules can only be isomorphic if $e_{1,2} \cdot 1$ is the same in both modules.
However, given two modules from the set \eqref{clasl22}, we see that $(e_{1,2} \cdot 1)$ is determined by its degree, its leading coefficient and its set of zeros. All these three properties
coincide nontrivially only in the pairs \[(\mathrm{F}_{(a,1)}(M'_{b}),\mathrm{F}_{(a,1)}(M'_{-b-1}))
       \text{ and } (\mathrm{F}_{(a,1)} \circ \mathrm{F}_{\tau}(M'_{b}), \mathrm{F}_{(a,1)} \circ \mathrm{F}_{\tau}(M'_{-b-1})).\]
But unless they coincide, precisely one of $b$ and $-b-1$ lies in the set $\bb{C}_{\geq -\frac{1}{2}}$. 
This shows that no nontrivial isomorphisms exists between the objects listed in \eqref{clasl22}.
\end{proof}

\subsection{Submodules and quotients}
Now we turn to some properties of the objects in $\ca{M}$.
\begin{lemma}
 \label{lemmazeroes}
 The modules $\mathrm{F}_{(a,1)}(M'_{b})$ and $\mathrm{F}_{\tau} \circ \mathrm{F}_{(a,1)}(M'_{b})$ from the classification in the previous section are simple for all $b \in \mathbb{C}$ and $a \in \bb{C}^{*}$.
\end{lemma}
\begin{proof}
Since $\mathrm{F}_{\tau}$ and $\mathrm{F}_{(a,1)}$ are auto-equivalences it suffices to prove that $M_{b}'$ is simple for all $b \in \bb{C}$.
Let $S$ be a nonzero submodule of $M'_{b}$ and let $f$ be a nonzero polynomial in $S$. We let $N(f)$ be the set of zeros of $f$; this is a finite subset of $\bb{C}$.
From the definition of the module structure we see that \[N(e_{1,2} \cdot f) = N(f)+1,\]
and inductively we obtain \[N(e_{1,2}^{k} \cdot f) = N(f)+k.\]
Now take $k$ large enough so that $(N(f)+k) \cap N(f)= \varnothing$. Then $f$ and $e_{1,2}^{k} \cdot f$ are relatively prime elements of $S$, and we can find $g_{1},g_{2} \in \bb{C}[h]$ such that
\[g_{1}f + g_{2}e_{1,2}^{k} \cdot f = 1 \in S.\] Then we have $S=\bb{C}[h] = M'_{b}$ which shows that $M'_{b}$ is simple for each $b \in \mathbb{C}$.
\end{proof}

The last type of modules is more interesting.

\begin{lemma}
\label{simpleproof}
\begin{enumerate}[$($i$)$]
 \item  For $2b \not\in \bb{N}_{0}$ the module $\mathrm{F}_{(a,1)}(M_{b})$ is simple.
 \item  For $2b\in\bb{N}_{0}$ the module $\mathrm{F}_{(a,1)}(M_{b})$ has a unique proper 
 submodule which is isomorphic to $\mathrm{F}_{(a,1)}(M_{-b-1})$,
 and the corresponding simple quotient is isomorphic to the simple finite dimensional module
 $L(2b)$ with highest weight $2b$ and dimension $2b+1$. In other words, we have a nonsplit short exact sequence:
\[0 \longrightarrow  \mathrm{F}_{(a,1)}(M_{-b-1}) \longrightarrow \mathrm{F}_{(a,1)}(M_{b}) \longrightarrow L(2b) \longrightarrow 0.\]
\end{enumerate}
\end{lemma}
\begin{proof}
We first prove the two statements for $M_{b}$ using an argument similar to that of Lemma~\ref{lemmazeroes}. We see that
\[N(e_{1,2} \cdot f) = \{-b\} \cup (N(f)+1),\]
so inductively we obtain
  \[N(e_{1,2}^{k} \cdot f) = \{-b,-b+1,-b+2, \ldots, -b+k-1\} \cup (N(f)+k).\]
Similarly,
\[N(e_{2,1} \cdot f) = \{b\} \cup (N(f)-1),\] which implies
\[N(e_{2,1}^{k} \cdot f) = \{b,b-1,b-2, \ldots, b-k+1\} \cup (N(f)-k).\]
Now for large integers $k$,  \[(N(f)-k) \cap \big( \{-b,-b+1,-b+2, \ldots, -b+k-1\} \cup (N(f)+k) \big) = \varnothing,\]
and  \[(N(f)+k) \cap \big( \{b,b-1,b-2, \ldots, b-k+1\} \cup (N(f)-k) \big) = \varnothing.\]

Note that for $2b \not\in \bb{N}_{0}$, we also have
\[\{-b,-b+1,-b+2, \ldots, -b+k-1\} \cap \{b,b-1,b-2, \ldots, b-l+1\} = \varnothing\] for all natural numbers $k$ and $l$, so as in the argument in Lemma~\ref{lemmazeroes}, 
we see that $e_{1,2}^{k} \cdot f$ and $e_{2,1}^{k} \cdot f$ are relatively prime for large enough $k$ so any submodule containing a nonzero polynomial $f$ also contains $1$ and the submodule is all of $M_{b}$.

Finally, suppose $2b \in \mathbb{N}_{0}$. We claim that 
\[S:=  \mathbb{C}[h]\prod_{j=0}^{2b}(h+b-j)\] is a proper submodule of $M_{b}$.
Clearly $\mathfrak{h} S \subset S$. We now calculate explicitly
\begin{align*}
 e_{1,2}\cdot \prod_{j=0}^{2b}(h+b-j) f =& (h+b)\sigma\prod_{j=0}^{2b}(h+b-j) f\\
=&(h+b)\prod_{j=0}^{2b}(h+b-j-1) \sigma f\\
=&(h+b)\prod_{k=1}^{2b+1}(h+b-k) \sigma f\\
=&((h-b-1)\sigma f)(h+b)\prod_{k=1}^{2b}(h+b-k)\\
=&((h-b-1)\sigma f)\prod_{k=0}^{2b}(h+b-k),
\end{align*}
which shows  that $e_{1,2}S \subset S$. Analogous calculations show that $e_{2,1}S \subset S$ and thus $S$ is a proper submodule.
 Now write $Q:=\prod_{j=0}^{2b}(h+b-j)$. Then $S=\{pQ \;| \; p \in \mathbb{C}[h]\}$. An explicit calculation gives
\begin{displaymath}
 \begin{array}{rcl}
h\cdot Qf&:=&Qhf \\
e_{1,2}\cdot Qf&:=&Q((h-b-1) \sigma f)  \\
e_{2,1}\cdot Qf&:=&Q(-(h+b+1) \sigma^{-1} f).  \\
\end{array}
\end{displaymath}
Thus we immediately see that $f \mapsto fQ$ is an isomorphism $M_{-b-1} \rightarrow S$. Note that $S$ is simple since $-b-1 \not\in \mathbb{N}_{0}$.
 Next we look at the quotient $M_{b} / S$. Define
\[v:=\prod_{j=0}^{2b-1}(h+b-j) + S.\]
Then $e_{1,2} \cdot v = 0$ and $(e_{1,1}-e_{2,2}) \cdot v = 2h \cdot v = 2bv$ so $v$ is a highest weight  vector of weight $2b$. Hence, since $\dim M_{b} / S =2b+1$,  $M_{b} / S$ is isomorphic to $L(2b)$.
We thus have a nonsplit short exact sequence:
\[0 \longrightarrow  M_{-b-1} \longrightarrow M_{b} \longrightarrow L(2b) \longrightarrow 0.\]
This proves the statements for $M_{b}$. 
Since the functor $\mathrm{F}_{(a,1)}$ is an auto-equivalence it maps simples to simples, it follows that
 $\mathrm{F}_{(a,1)}(M_{b})$ is simple for $2b \not\in \bb{N}_{0}$. Application of the exact functor $\mathrm{F}_{(a,1)}$ to our short exact sequence, we get the corresponding sequence
\[0 \longrightarrow  \mathrm{F}_{(a,1)}(M_{-b-1}) \longrightarrow \mathrm{F}_{(a,1)}(M_{b}) \longrightarrow \mathrm{F}_{(a,1)}(L(2b)) \longrightarrow 0,\] and since $\mathrm{F}_{(a,1)}$ is the identity functor on
 finite dimensional modules by Lemma~\ref{functorlemma}\eqref{fun5}, we have an exact sequence
\[0 \longrightarrow  \mathrm{F}_{(a,1)}(M_{-b-1}) \longrightarrow \mathrm{F}_{(a,1)}(M_{b}) \longrightarrow L(2b) \longrightarrow 0,\]
as claimed.
\end{proof}

\begin{rmk}
The above lemma shows that every finite dimensional simple $\mathfrak{sl}_{2}$-module can be expressed as a quotient of two (infinite dimensional) modules from $\ca{M}$.
This is similar to what holds for Verma modules: if $2b\in \bb{N}_{0}$ there exists a short nonsplit exact sequence
\[0 \longrightarrow  M(-2b-2) \longrightarrow M(2b) \longrightarrow L(2b) \longrightarrow 0,\]
where $M(\lambda)$ is the Verma module of highest weight $\lambda$.
\end{rmk}

\subsection{Central character}
The simple modules $\mathrm{F}_{(a,1)}(M_{b}), \mathrm{F}_{(a,1)}(M_{b}')$ and $\mathrm{F}_{\tau} \circ \mathrm{F}_{(a,1)}(M_{b}')$ have central characters by Schur's lemma, that is 
there exist algebra homomorphisms
$\chi_{N}: Z(\mathfrak{sl}_{2}) \longrightarrow \bb{C}$ such that 
\[z \cdot v = \chi_{N}(z)v\quad \text{ for all }  z\in Z(\mathfrak{sl}_{2}) \text{ and } v \in N,\]
for each $N \in \{ \mathrm{F}_{(a,1)}(M_{b}), \mathrm{F}_{(a,1)}(M_{b}'),\mathrm{F}_{\tau} \circ \mathrm{F}_{(a,1)}(M_{b}')\}$.
\begin{prop}
Let $N \in \{ \mathrm{F}_{(a,1)}(M_{b}), \mathrm{F}_{(a,1)}(M_{b}'),\mathrm{F}_{\tau} \circ \mathrm{F}_{(a,1)}(M_{b}')\}$.
 The central character of $N$ is determined by
\[\chi_{M_{b}}(c_{2}) = \chi_{M_{b}'}(c_{2}) = \chi_{M_{b}''}(c_{2}) = 2b(b+1),\]
where $c_{2}=2h^{2} +e_{1,2}e_{2,1}+e_{2,1}e_{1,2}$. 
\end{prop}
\begin{proof}
Since the center of $\mathcal{U}(\mathfrak{sl}_{2})$ is $\bb{C}[c_{2}]$, the central character of a module is determined by the single scalar 
\begin{equation}
\label{eqchar}
 \chi(c_{2}) = c_{2} \cdot 1 = (2h^{2} +e_{1,2}e_{2,1}+e_{2,1}e_{1,2}) \cdot 1.
\end{equation}
Since we know that the right hand side of \eqref{eqchar} is a scalar, we need only consider the constant term on the left hand side of \eqref{eqchar}.
An explicit calculation gives the central characters as stated in the proposition.
\end{proof}

\subsection{Tensoring with finite dimensional modules}
\label{tensoring}
Let $L$ be the natural $\mathfrak{sl}_{2}$-module. It has basis $\{e_{1},e_{2}\}$ and the action is given by
\[e_{i,j} \cdot e_{k} = \delta_{j,k}e_{i}.\]
In the basis  $\{e_{1},e_{2}\}$ we explicitly have:
\begin{displaymath}
 \begin{array}{rclrcl}
   h \cdot e_{1} &=& \frac{1}{2}e_{1}  \qquad&    h \cdot e_{2} &=& -\frac{1}{2}e_{2}  \\
   e_{1,2} \cdot e_{1} &=& 0  &      e_{1,2} \cdot e_{2} &=& e_{1}  \\  
   e_{2,1} \cdot e_{1} &=& e_{2}  &     e_{2,1} \cdot e_{2} &=& 0.             
 \end{array}
\end{displaymath}
Being simple and $2$-dimensional, $L$ is isomorphic to $L(1)$, the simple highest weight module of highest weight $1$, and from here on we shall identify the two. 
Let $N \in \{M,M'\}$ and consider the module $N_{b} \otimes L(1)$. A basis for $N_{b} \otimes L(1)$ is $\{h^{k} \otimes e_{1}| k \geq 0\} \cup \{h^{k} \otimes e_{2}| k \geq 0\}$,
so, in particular, every element of $N_{b} \otimes L(1)$ has a unique expression of the form
\[(f,g) := f \otimes e_{1} + g\otimes e_{2},\] for some $f,g\in \bb{C}[h]$. In this notation, the action of $\mathfrak{sl}_{2}$ on $N_{b} \otimes L(1)$ is given by the following lemma which is easily proved by a 
straightforward computation.

\begin{lemma}
 The action of  $\mathfrak{sl}_{2}$ on $M_{b} \otimes L(1)$ is given by:
\begin{displaymath}
 \begin{array}{rcl}
   h \cdot (f,g) &=& \big( (h+\frac{1}{2})f,(h-\frac{1}{2})g \big),  \\
   e_{1,2} \cdot (f,g) &=& \big( (h+b)\sigma(f) + g, (h+b)\sigma(g) \big),  \\         
   e_{2,1} \cdot (f,g) &=& \big( -(h-b)\sigma^{-1}(f), -(h-b)\sigma^{-1}(g) + f \big) .
 \end{array}
\end{displaymath}
 The action of  $\mathfrak{sl}_{2}$ on $M_{b}' \otimes L(1)$ is given by:
\begin{displaymath}
 \begin{array}{rcl}
   h \cdot (f,g) &=& \big((h+\frac{1}{2})f,(h-\frac{1}{2})g\big),  \\
   e_{1,2} \cdot (f,g) &=& \big( \sigma(f) + g, \sigma(g)\big) , \\         
   e_{2,1} \cdot (f,g) &=& \big(-(h+b+1)(h-b)\sigma^{-1}(f), -(h+b+1)(h-b)\sigma^{-1}(g) + f\big) .
 \end{array}
\end{displaymath}
\end{lemma}

We can now derive an explicit decomposition formula.
\begin{prop}
\label{propdecomp}
 Let $N \in \{M,M'\}$.
\begin{enumerate}[$($i$)$]
 \item\label{propd1} For $2b \neq -1$, we have
   \[N_{b}\otimes L(1) \simeq N_{b-\frac{1}{2}} \oplus N_{b+\frac{1}{2}}.\]
 \item\label{propd2} For $2b=-1$, we have a nonsplit short exact sequence 
   \[0 \longrightarrow N_{0} \longrightarrow N_{-\frac{1}{2}} \otimes L(1) \longrightarrow N_{-1} \longrightarrow 0.\]
\end{enumerate}
\end{prop}
\begin{proof}
We first consider the case $N = M'$.
We determine the submodules of $M_{b}' \otimes L(1)$. Let $S$ be a nonzero submodule containing a nonzero element $(g_{1},g_{2})$.
Assume first that $g_{2} \neq 0$. Since
\[(e_{1,2}-1) \cdot (g_{1},g_{2}) = ( \sigma(g_{1}) + g_{2} - g_{1}, \sigma(g_{2})-g_{2}),\]
by Lemma~\ref{deg}, we see that $(e_{1,2}-1)$ acts by decreasing the degree by $1$ in the second component. Thus $c(e_{1,2}-1)^{\deg g_{2}} \cdot (g_{1},g_{2}) = (f,1) \in S$,
for some $f \in \bb{C}[h]$ and some $c \in \bb{C}^{*}$. 
Acting again by $(e_{1,2}-1)$ we obtain $(\sigma(f)-f+1,0) \in S$.
Assume $\sigma(f)-f+1 \neq 0$. Then, noting that 
 $(e_{1,2}-1) \cdot (g,0) = ( \sigma(g) - g,0)$,
we obtain
 \[c(e_{1,2}-1)^{\deg \sigma(f)-f+1} \cdot (\sigma(f)-f+1,0) = (1,0) \in S,\]
for some $c \in \bb{C}^{*}$. 
But then, since $1$ and $(h+\frac{1}{2})$ generate the algebra $\bb{C}[h]$, by acting on $(1,0)$ by $\bb{C}[h]$ we obtain
$(\bb{C}[h],0) \subset S$. Since $(g_{1},g_{2}) \in S$, we also obtain $(0,g_{2})\in S$ which as above gives us $(0,1) \in S$ and $(0,\bb{C}[h]) \subset S$.
 Thus $M_{b}' \otimes N = (\bb{C}[h],\bb{C}[h]) \subset S$ and $S=M_{b}' \otimes N$.
On the other hand, if $g_{2}=0$ to begin with, by the same argument we immediately get $(1,0) \in S$ and $(\bb{C}[h],0) \subset S$. Now $e_{2,1} \cdot (1,0) = (-(h+b+1)(h-b),  1) \in S$
 so we again obtain $(0,1) \in S$ and $(0,\bb{C}[h]) \subset S$, so again $S=M_{b}' \otimes N$.

The only case remaining is when $(f,1) \in S$ where $\sigma(f)-f+1=0$.
The equation $\sigma(f)-f=-1$
 has the form discussed in Lemma~\ref{lemmasigmaeq} and we know how to solve it. The solutions are precisely $f= h + c$,
 where $c$ is some constant. Explicit calculations show that the submodule generated by $(h+c,1)$ is proper if and only if $c \in \{-b,b+1\}$. Thus we define $S_{1}$ to be the submodule of $M_{b}'$
 generated by $(h-b,1)$, and we define $S_{2}$ to be the submodule of $M_{b}'$ generated by $(h+b+1,1)$.

We define two linear maps $\Phi_{1}: \ca{P} \rightarrow S_{1}$ and $\Phi_{2}: \ca{P} \rightarrow S_{2}$ by
\[\Phi_{1}(f(h)):= \big( (h-b)f(h+\frac{1}{2}),f(h-\frac{1}{2}) \big) \text{ and } \Phi_{2}(f(h)):= \big( (h+b+1)f(h+\frac{1}{2}),f(h-\frac{1}{2}) \big).\]
Explicit calculations show that
\begin{displaymath}
 \begin{array}{rcl}
   h \cdot 	 \Phi_{1}(f) &=& \Phi_{1}(hf)  \\
   e_{1,2}\cdot  \Phi_{1}(f) &=& \Phi_{1}(\sigma f) \\         
   e_{1,2} \cdot \Phi_{1}(f) &=& \Phi_{1}(-(h+(b-\frac{1}{2})+1)(h-(b-\frac{1}{2}))\sigma^{-1} f),
 \end{array}
\end{displaymath}
and also
\begin{displaymath}
 \begin{array}{rcl}
   h \cdot 	 \Phi_{2}(f) &=& \Phi_{2}(hf)  \\
   e_{1,2}\cdot  \Phi_{2}(f) &=& \Phi_{2}(\sigma f) \\         
   e_{1,2} \cdot \Phi_{2}(f) &=& \Phi_{2}(-(h+(b+\frac{1}{2})+1)(h-(b+\frac{1}{2}))\sigma^{-1} f).
 \end{array}
\end{displaymath}
This shows that $\Phi_{1}$ is an isomorphism $M_{b-\frac{1}{2}}' \rightarrow S_{1}$ and that $\Phi_{2}$ is an isomorphism $M_{b+\frac{1}{2}}' \rightarrow S_{2}$.
Finally, we show that for $b \neq -\frac{1}{2}$, the submodules $S_{1}$ and $S_{2}$ are complementary. If $(f,g) \in S_{1} \cap S_{2}$, then $(h-b)\sigma^{-1}(g)=f=(h+b+1)\sigma^{-1}(g)$,
so if $2b+1 \neq 0$ we nessecarily have $f=g=0$ and $S_{1} \cap S_{2} = 0$.

Moreover, we have $S_{1} + S_{2} \ni (h+b+1,1)-(h-b,1) = (2b+1,0)$, so if $2b+1 \neq 0$, $S_{1} + S_{2} \supset (\ca{P},0)$ which then gives $(0,1) \in S_{1} + S_{2}$,  $(0,\ca{P}) \subset S_{1} + S_{2}$, 
 and $S_{1} + S_{2} = (\ca{P},\ca{P}) = M_{b}' \otimes N$.

Thus, if $2b+1 \neq 0$ we have \[M_{b}' \otimes N  = S_{1} \oplus S_{2} \simeq M_{b-\frac{1}{2}}' \oplus M_{b+\frac{1}{2}}',\]
as claimed in \eqref{propd1}.

On the other hand, if $2b+1=0$, then $S_{1} = S_{2}$ is the unique submodule of $M_{b}' \otimes N$, and explicit calculations show that the quotient module also is isomorphic to $S_{1}$ and $S_{2}$.
Thus in this case we have a nonsplit exact sequence in $\ca{M}$:
\[0 \longrightarrow M_{0}' \longrightarrow M_{-\frac{1}{2}}' \otimes N \longrightarrow M_{-1}' \longrightarrow 0.\]
Thus \eqref{propd1} and \eqref{propd2} are proved for $N_{b}=M_{b}'$. 

For $N_{b}= M_{b}$ take instead $S_{1}:= \langle (1,1) \rangle$ and $S_{2}:= \langle (h-b,h+b) \rangle$ in $M_{b} \otimes L(1)$. Then
 \[f(h) \mapsto (f(h+\frac{1}{2}),f(h-\frac{1}{2}))\] 
is an isomorphism $M_{b+\frac{1}{2}} \rightarrow S_{1}$, and  
\[f(h) \mapsto ((h-b)f(h+\frac{1}{2}),(h+b)f(h-\frac{1}{2}))\]
is an isomorphism $M_{b-\frac{1}{2}} \rightarrow S_{2}$, and for $b \neq -\frac{1}{2}$ we have $M_{b} = S_{1} \oplus S_{2}$.
On the other hand, for $b=-\frac{1}{2}$ the sum is not direct, but we still have $(M_{-\frac{1}{2}} \otimes L(1)) \slash S_{1} \simeq M_{-1}$, so by Lemma~\ref{simpleproof},
 the multiset of Jordan-H\"{o}lder components of $M_{-\frac{1}{2}} \otimes L(1)$ is precisely
 $\{L(0),M_{-1},M_{-1}\}$.
\end{proof}

We can now describe explicitly the decomposition of modules obtained by taking the tensor product of a module from $\ca{M}$ with a finite dimensional module.
For this it suffices to consider only the simple finite dimensional module $L(k)$.

\begin{rmk}
\label{rmkGG}
The Clebsch-Gordan formula for simple finite dimensional $\mathfrak{sl}_{2}$-modules is well known, see for example Theorem~$1.39$ in \cite{Ma}. It states states that for $m,n \in \bb{N}_{0}$ with $m \geq n$ we have
\[L(m) \otimes L(n) = L(m+n) \oplus L(m+n-2) \oplus  \cdots \oplus L(m-n).\]
\end{rmk}

\begin{cor}
For all $2b \in \bb{C} \setminus \bb{N}_{0}$ and for $N_{b} \in \{M_{b},M_{b}'\}$ we have
\[N_{b} \otimes L(k) \simeq \bigoplus_{i=0}^{k} N_{b+\frac{k-2i}{2}}.\]
\end{cor}
\begin{proof}
We proceed by induction on $k$, the case $k=0,1$ holding by Proposition~\ref{propdecomp}.
Assume the claim of the Corollary holds for $k$ and for $k-1$ and apply $-\otimes L(1)$ to both sides of $N_{b} \otimes L(k) \simeq \bigoplus_{i=0}^{k} N_{b+\frac{k-2i}{2}}$.
 Using associativity of the tensor product and applying the Clebsch-Gordan formula from Remark~\ref{rmkGG} on the left, we obtain
\[N_{b} \otimes (L(k+1) \oplus L(k-1)) \simeq \big( \bigoplus_{i=0}^{k} N_{b+\frac{k-2i}{2}} \big) \otimes L(1).\]
Using the distributive propery of the tensor product, this simplifies to
\[(N_{b} \otimes L(k+1)) \oplus (N_{b} \otimes L(k-1)) \simeq  \bigoplus_{i=0}^{k}( N_{b+\frac{k-2i}{2}} \otimes L(1)).\]
Now, since $b+\frac{k-2i}{2} \not\in \bb{Z}$ for all $i,k \in \bb{Z}$, we can apply our formula to both sides giving
\[(N_{b} \otimes L(k+1)) \oplus \bigoplus_{i=0}^{k-1} N_{b+\frac{k-1-2i}{2}} \simeq  \bigoplus_{i=0}^{k} N_{b+\frac{k-2i-1}{2}} \oplus \bigoplus_{i=0}^{k} N_{b+\frac{k-2i+1}{2}}.\]
Writing \[ \bigoplus_{i=0}^{k} N_{b+\frac{k-2i-1}{2}} = N_{b+\frac{-k-1}{2}} \oplus \bigoplus_{i=0}^{k-1} N_{b+\frac{k-2i-1}{2}},\]
and inserting this in the above formula, we obtain
\[(N_{b} \otimes L(k+1)) \oplus \bigoplus_{i=0}^{k-1} N_{b+\frac{k-1-2i}{2}} \simeq   N_{b+\frac{-k-1}{2}} \oplus \bigoplus_{i=0}^{k-1} N_{b+\frac{k-2i-1}{2}} \oplus \bigoplus_{i=0}^{k} N_{b+\frac{k-2i+1}{2}}.\]
Now we can cancel corresponding equal direct summands on each side resulting in
\[N_{b} \otimes L(k+1)  \simeq   N_{b+\frac{-k-1}{2}} \oplus \bigoplus_{i=0}^{k} N_{b+\frac{k-2i+1}{2}} \simeq \bigoplus_{i=0}^{k+1} N_{b+\frac{(k+1)-2i}{2}}.\]
Thus the formula holds for $k+1$, and the claim of the corollary follows by induction.
\end{proof}

\section{The $\mathfrak{sl}_{n+1}$ case}
\label{sln}
We now try to generalize the above results to the general case. We are trying to find all $(p_{1}, \ldots, p_{n},q_{1}, \ldots, q_{n}) \in \ca{P}^{2n}$ such that $\ca{P}$ becomes an $\g$-module under the action:
\begin{displaymath}
\begin{array}{rcll}
 h_{k} \cdot f &=& h_{k} f &  k \in \oln,\\
 e_{i,n+1} \cdot f &=& p_{i} \sigma_{i} f  &  i \in \oln,\\
 e_{n+1,j} \cdot f &=& q_{i} \sigma_{j}^{-1} f  &  j \in \oln,\\
 e_{i,j} \cdot f &=& \big( p_{i}\sigma_{i}(q_{j})-q_{j}\sigma_{j}^{-1}(p_{i})\big) \sigma_{i}\sigma_{j}^{-1}f  \qquad&  i,j \in \oln, i \neq j.\\
\end{array}
\end{displaymath}

From here on, assume that $(p_{1}, \ldots, p_{n},q_{1}, \ldots, q_{n})$ yields a module. We shall find a number of necessary relations amongst these polynomials.

\subsection{Technical Lemmas}
\begin{lemma}
\label{lemma3e}
For all $i,j\in \oln$ we have
\begin{enumerate}[$($i$)$]
 \item\label{ll1} $p_{i}\sigma_{i}(p_{j})= p_{j}\sigma_{j}(p_{i})$,
 \item\label{ll2} $q_{i}\sigma_{i}^{-1}(q_{j})= q_{j}\sigma_{j}^{-1}(q_{i})$,
 \item\label{ll3} $\sigma_{i}^{-1}(p_{i})q_{i} = -h_{i}(\overline{h}+1) + \tilde{g}_{i}$, where $\tilde{g}_{i} \in \ca{P}_{i}$ for each $i$.
\end{enumerate}
\end{lemma}
\begin{proof}
 Statement \eqref{ll1} and \eqref{ll2} are equivalent to the two identities
 \[e_{i,n+1} \cdot e_{j,n+1} \cdot f - e_{j,n+1} \cdot e_{i,n+1} \cdot f = [e_{i,n+1},e_{j,n+1}] \cdot f = 0\] and
 \[e_{n+1,i} \cdot e_{n+1,j} \cdot f - e_{n+1,j} \cdot e_{n+1,i} \cdot f  =[e_{n+1,i},e_{n+1,j}]\cdot f = 0,\]
 for $i,j \in \oln$.

 For \eqref{ll3}, consider the identity 
\begin{equation}
\label{eq2}
 [e_{i,n+1},e_{n+1,i}] \cdot f = e_{i,n+1} \cdot e_{n+1,i} \cdot f - e_{n+1,i} \cdot e_{i,n+1} \cdot f.
\end{equation}
 Using our explicit choice of basis in $\ca{P}$, we have $e_{i,i}-e_{n+1,n+1} = h_{i} + \olh$ so \eqref{eq2} becomes
\[(h_{i} + \olh) \cdot f = p_{i}\sigma_{i}(q_{i}) \cdot f - q_{i}\sigma_{i}^{-1}(p_{i}) \cdot f,\] or, equivalently,
\[h_{i} + \olh = p_{i}\sigma_{i}(q_{i}) - q_{i}\sigma_{i}^{-1}(p_{i}).\] Substituting $f:=q_{i}\sigma_{i}^{-1}(p_{i})$, it reads $\sigma_{i}(f)-f = h_{i} + \olh$.
This equation is of the form discussed in Lemma~\ref{lemmasigmaeq} so we know how to solve it. The set of solutions is precisely
 $\{ -h_{i}(\olh+1) + \tilde{g}_{i} \; | \; \tilde{g}_{i} \in \ca{P}_{i} \}$,
 as claimed in the lemma.
\end{proof}

\begin{rmk}
 \label{divcor}
Note that claims \eqref{ll1} and \eqref{ll2} of Lemma~\ref{lemma3e} are equivalent to
\begin{equation}
 \label{eq4}
p_{i}(\sigma_{i}(p_{j})-p_{j})= p_{j}(\sigma_{j}(p_{i})-p_{i})
\end{equation}
and
\begin{equation}
 \label{eq5}
q_{i}(\sigma_{i}^{-1}(q_{j})-q_{j})= q_{j}(\sigma_{j}^{-1}(q_{i})-q_{i}).
\end{equation}
\end{rmk}

\begin{lemma}
\label{lemmatech}
The polynomials $p_{1}, \ldots, p_{n},q_{1}, \ldots, q_{n}$ satisfy the following:
 \begin{enumerate}[$($i$)$]

  \item\label{tec1} $\deg_{i} p_{i}, \deg_{i}q_{i} \in \{0,1,2\}$.

  \item\label{tec2} $\deg_{i} p_{i} + \deg_{i}q_{i}=2.$

  \item\label{tec3} If $deg_{k}p_{k} = 2$, then $p_k$ has a nontrivial factorization. Similarly for $q$.

  \item\label{tec4} If $deg_{k}p_{k} = 1$, then $p_k$  is irreducible. Similarly for $q$.

  \item\label{tec5} Suppose $\deg_{k}p_{k}=1$ and $\deg_{i} p_{i}=2$. Then $p_{k}$ divides $p_{i}$. Similarly for $q$.

  \item\label{tec6} Let $S=\{\deg_{i} p_{i} \;|\; i \in \oln\}$. Then either $S \subset \{0,1\}$ or $S \subset \{1,2\}$. Similarly for $q$.

  \item\label{tec7} Suppose  $\deg_{k}p_{k}= 2 =\deg_{i} p_{i}$. Then $p_{i}$ and $p_{k}$ share a common factor. Similarly for $q$.

 \end{enumerate}
\end{lemma}
\begin{proof}
 First note that since $\tau$ is an algebra automorpism of $\uh$, by Proposition~\ref{symmprop} it suffices to prove statements \eqref{tec3} - \eqref{tec7} for the polynomials $p$.

 By part \ref{ll3} of Lemma~\ref{lemma3e} we have $p_{i},q_{i} \neq 0$. Applying $\deg_{i}$ to the same equality we get
\[\deg_{i}(p_{i}) + \deg_{i}(q_{i}) = \deg_{i}(p_{i}q_{i})=\deg_{i}(\sigma_{i}^{-1}(p_{i})q_{i})=\deg_{i}(-h_{i}(\olh + 1))=2,\]
which proves claims \eqref{tec1} and \eqref{tec2}.

We now look at claim \eqref{tec3}. The case $n=1$ is obvious so let $n \geq 2$. Assume that $\deg_{k}p_{k}=2$ with $p_{k}$ irreducible. Consider the equality $\sigma_{k}^{-1}(p_{k})q_{k} = -h_{k}(\olh + 1) + \tilde{g}_{k}$
 from Lemma~\ref{lemma3e}. By comparing the coefficients of $h_{k}^{2}$ (or by application of $\mathtt{c}_{k}$) on both sides, we see that $q_{k}=c \in \bb{C}^{*}$ so we have
\[(p_{k},q_{k}) = (-c^{-1}((h_{k}-1)\olh - \tilde{g}_{k}), c).\]
Now let $i$ be an index different from $k$. We have
 $p_{i}(\sigma_{i}(p_{k})-p_{k})= p_{k}(\sigma_{k}(p_{i})-p_{i})$ from formula  \eqref{eq4}.
Since $p_{k}$ is irreducible, it divides one of the factors on the left. However, $p_{k}$ does not divide $\sigma_{i}(p_{k})-p_{k}$ since the latter is a nonzero polynomial with lower $i$-degree than $p_{k}$.
Thus $p_{k}$ divides $p_{i}$, which implies that $\sigma_{i}^{-1}p_{k}$ divides $\sigma_{i}^{-1}p_{i}$ which, in turn, divides $-h_{i}(\olh + 1) + \tilde{g}_{i}$. Thus we have
\begin{equation}
 \label{eq3}
(h_{k}(\olh+1)-\tilde{g}_{k}) f = -h_{i}(\olh + 1) + \tilde{g}_{i}
\end{equation}
for some $f \in \ca{P}$. On the right hand side of \eqref{eq3} we have the term $-h_{i}^{2}$ which comes from the product $\tilde{g}_{k}f$. However, if $\deg_{i} f>0$, then we get terms
 of the form $h_{k}h_{i}^{1+\deg_{i} f}$ on the left of \eqref{eq3} which does not appear on the right. Thus $\deg_{i} f = 0$ and $\deg_{i} \tilde{g}_{k}=2$ and applying $\mathtt{c}_{i}$ to \eqref{eq3} gives
 $\mathtt{c}_{i}(\tilde{g}_{k})\mathtt{c}_{i}(f) = -1$.
In particular, this shows that $\mathtt{c}_{i}(f)=f$ is invertible and we have $p_{i}=ap_{k}$ for some $a \in \bb{C}^{*}$. But then $p_{i}\sigma_{i}(p_{k}) = p_{k}\sigma_{k}(p_{i})$ simplifies to
 $\sigma_{i}(p_{k}) = \sigma_{k}(p_{k})$ which contradicts the form of $p_{k}$ above. Thus $p_{k}$ is reducible. The argument for $q_{k}$ is analogous.

For claim \eqref{tec4}, let $\deg_{k}p_{k}=1$. Then by \eqref{tec2} we also have $\deg_{k}q_{k}=1$ and thus
\[(p_{k},q_{k}) = (f_{1}h_{k}+g_{1}, f_{2}h_{k}+g_{2})\]
for some $f_{1},f_{2},g_{1},g_{2} \in \ca{P}_{k}$. But then, since $\sigma_{k}^{-1}(p_{k})q_{k} = -h_{k}(\olh + 1) +\tilde{g}_{k}$ and the coefficient of $h_{k}^{2}$ on the right is $-1$, we have
 $f_{1},f_{2} \in \bb{C}^{*}$ and then, clearly, $p_{k}$ and $q_{k}$ are both irreducible.

To prove claim \eqref{tec5} we consider again the equality $p_{i}(\sigma_{i}(p_{k})-p_{k})= p_{k}(\sigma_{k}(p_{i})-p_{i})$ given by formula \eqref{eq4}.
Since $\deg_{k} p_{k}=1$, the polynomial $p_{k}$ is irreducible by claim \eqref{tec4} so $p_{k}$ divides either $p_{i}$ or $(\sigma_{i}(p_{k})-p_{k})$. However, considering the $i$-degree, we have that 
$p_{k} | (\sigma_{i}(p_{k})-p_{k})$ only if $(\sigma_{i}(p_{k})-p_{k})=0$. But the right hand side of \eqref{eq4} is nonzero since we know that $p_{i}$ has the form $-\frac{1}{c}((h_{i}-1)\olh + \tilde{g}_{i})$.
Thus the only remaining possibility is that $p_{k} | p_{i}$.

To prove claim \eqref{tec6} we suppose that there exist indices $i,k$ such that $\deg_{i}p_{i} = 2$ and $\deg_{k}p_{k}=0$. Then, as in the proof of claim \eqref{tec3}, we know that $q_{i}$ is a constant and $\deg_{k} q_{k}= 2$.
But then the equation $q_{i}\sigma_{i}^{-1}(q_{k})= q_{k}\sigma_{k}^{-1}(q_{i})$
 from Lemma~\ref{lemma3e} simplifies to $\sigma_{i}^{-1}(q_{k})= q_{k}$,
which does not hold since $q_{k}$ depends on $i$.

We now turn to claim \eqref{tec7}. Let $p_{i}=\alpha_{1}\alpha_{2}$ and $p_{k}=\beta_{1}\beta_{2}$ be the corresponding decompositions into prime polynomials. Then the equation
 $p_{i}\sigma_{i}(p_{k})= p_{k}\sigma_{k}(p_{i})$ from Lemma~\ref{lemma3e}
 is equivalent to \[\alpha_{1}\alpha_{2}\sigma_{i}(\beta_{1}\beta_{2}) = \beta_{1}\beta_{2}\sigma_{k}(\alpha_{1}\alpha_{2}).\]
Suppose $p_{i}$ and $p_{k}$ does not share a common factor. Then we have $\alpha_{1}\alpha_{2} = c\sigma_{k}(\alpha_{1}\alpha_{2})$ for some $c \in \bb{C}^{*}$.
 By applying $\mathtt{c}_{i}$ to both sides, we obtain $c=1$ and
 $p_{i}=\sigma_{k}p_{i}$ which is not
 possible since $p_{i}$ has the form $c(-h_{i}\olh + \tilde{g}_{i})$ 
 where $\tilde{g}_{i}\in\mathcal{P}_i$ and thus depends on $k$. Therefore $p_{i}$ and $p_{k}$ share a common factor.
\end{proof}

\begin{lemma}
\label{lemmalc}
With respect to the grading $\deg_{i}$, the leading coefficients of both $p_{i}$ and $q_{i}$ are invertible, 
that is $\mathtt{c}_{i}(p_{i}), \mathtt{c}_{i}(q_{i}) \in \mathbb{C}^{*}$.
\end{lemma}
\begin{proof}
 By Lemma~\ref{lemma3e}
 we have 
 $\sigma_{i}^{-1}(p_{i})q_{i} = -h_{i}(\overline{h}+1) + \tilde{g}_{i}$ 
 for some $\tilde{g}_{i}\in\mathcal{P}_i$. Applying $\mathtt{c}_{i}$ to this, we get
\[ \mathtt{c}_{i}(\sigma_{i}^{-1}(p_{i}))\mathtt{c}_{i}(q_{i}) = -1, \]
and $\mathtt{c}_{i}(\sigma_{i}^{-1}(p_{i})) = \mathtt{c}_{i}(p_{i})$ which shows that $\mathtt{c}_{i}(p_{i}), \mathtt{c}_{i}(q_{i}) \in \mathbb{C}^{*}$ as stated.
\end{proof}

\begin{rmk}
 \label{lead1rmk}
Note that, if a module is determined by $p_{1}, \ldots, p_{n}, q_{1}, \ldots, q_{n}$, we can apply the functor $\mathrm{F}_{a}$ with
\[a=(\mathtt{c}_{1}(p_{1})^{-1},\mathtt{c}_{2}(p_{2})^{-1}, \ldots, \mathtt{c}_{n}(p_{n})^{-1},1)\] to obtain a module where the leading coefficient of $p_{i}$ is $1$ for all $i\in \oln$.
Thus from here on we will assume that the leading coefficient of each $p_{i}$ is $1$. All other module structures can then be obtained by applying the functors $\mathrm{F}_{a}$.
\end{rmk}

\begin{lemma}For each $i \in \oln$ we have:
\label{lemmapform}
\begin{enumerate}[$($i$)$]
 \item\label{pform1} The irreducible components of $p_{i}$ have the form $h_{i} + \beta_{i}$
where $\beta_{i} \in \ca{P}_{i}$.
 \item\label{pform2} $\beta_{i} = \sum_{j \neq i} c_{j}^{(i)}h_{j} + c_{0}^{(i)}$ for some constants $c_{j}^{(i)}$.
 \item\label{pform3} $c_{j}^{(i)} \in \{0,1\}$ for all $i \in \oln$, $j \in \oln \setminus \{i\}$.
\end{enumerate}
\end{lemma}
\begin{proof}
 If $deg_{i} p_{i}$ is $0$ or $1$, then claim \eqref{pform1} follows from Lemma~\ref{lemmalc}. If  $deg_{i} p_{i} = 2$, then 
by Lemma~\ref{lemmatech}\eqref{tec3} we know that
 $\sigma_{i}^{-1}(p_{i})$ and thus also $p_{i}$ has a nontrivial factorization, say $p_{i}=fg$. But then $2 = \deg_{i}(f)+\deg_{i}(g)$, and
 $1 = \mathtt{c}_{i}(f)\mathtt{c}_{i}(g)$, which shows that $\deg_{i}(f) = 1 = \deg_{i}(g)$ since the factorization was nontrivial, and thus both $f$ and $g$ have the prescribed form.

To prove \eqref{pform2}, we need only show that for each $k \in \oln \setminus \{i\}$ we either 
have $\deg_{k} p_{i}=0$ or we have $\deg_{k} p_{i}=1$ and $\mathtt{c}_{k}(p_{i}) \in \bb{C}$.
 Suppose first that $\deg_{i}p_{i}=1$, so that $p_{i}=h_{i}+\beta_{i}$. We consider equation \eqref{eq4}:
\[p_{i}(\sigma_{i}(p_{k})-p_{k})= p_{k}(\sigma_{k}(p_{i})-p_{i}).\]
 If $\deg_{k} p_{k}=0$, then $p_{k}=1$ and \eqref{eq4} reads $0=\sigma_{k}(p_{i})-p_{i}$. This shows that $\deg_{k}p_{i}=0$ and we are done. If $\deg_{k} p_{k}=1$, then claim \eqref{pform1} and equation \eqref{eq4} give
\[(h_{i} + \beta_{i})(\sigma_{i}(\beta_{k})-\beta_{k})=(h_{k} + \beta_{k})(\sigma_{k}(\beta_{i})-\beta_{i}).\]
Since $h_{i} + \beta_{i}$ is prime by Lemma~\ref{lemmatech}, this term either divides the factor $(\sigma_{k}(\beta_{i})-\beta_{i})$, 
which implies $(\sigma_{k}(\beta_{i})-\beta_{i})=0$ and thus $\deg_{k} p_{i} = \deg_{k} \beta_{i} = 0$ and we are done;
 or $h_{i} + \beta_{i}$ divides $h_{k} + \beta_{k}$ which is also prime and thus $h_{i} + \beta_{i} = c(h_{k}+\beta_{k})$ for some $c\in \bb{C}^{*}$. Define
\[\gamma:= h_{i}-c\beta_{k} = ch_{k}-\beta_{i}.\] Then $\gamma \in \ca{P}_{i} \cap \ca{P}_{k}$, and we have $\beta_{i} = ch_{k} - \gamma$ as required.

 If $\deg_{k} p_{k}=2$, then by Lemma~\ref{lemmatech}\eqref{tec5} we know that $p_{k}$ shares a common factor with $p_{i}$, and again we have $h_{i} + \beta_{i} = c(h_{k}+\beta_{k})$
 for some nonzero $c$ and the same argument works. 
Suppose now instead that $\deg_{i}p_{i}=2$ and let $(h_{i}+\beta_{i})$ be a factor of $p_{i}$. Then we have
$p_{i}=(h_{i}+\beta_{i})(h_{i}+\overline{\beta}_{i})$,
where by Lemma~\ref{lemma3e} we have $q_{i}=-1$, and 
\[\overline{\beta}_{i} = \sum_{j\neq i}h_{j}-\beta_{i}-1.\]
Let $\deg_{k} p_{k}$ have degree $1$ or $2$. Then by Lemma~\ref{lemmatech} we know that $p_{k}$
 shares a common factor with $p_{i}$ and we again get an equality of the form $h_{i} + \beta_{i} = c(h_{k}+\beta_{k})$ for some $c \in \bb{C}^{*}$. By the same argument as above we again see that $\beta_{i}$ has the stated form.
But $\beta_{i}$ has the prescribed form if and only if $\overline{\beta}_{i}$ does, so we are done.

To prove claim \eqref{pform3} we fix an index $i\in \oln$. We shall prove that $c_{k}^{(i)} \in \{0,1\}$ for each $k \in \oln \setminus \{i\}$.
Suppose first that $\deg_{i} p_{i} = 1$. Then we have
\[p_{i} = (h_{i} + \sum_{j \neq i} c_{j}^{(i)}h_{j} + c_{0}^{(i)})\]
and, by Lemma~\ref{lemma3e}, we have
\begin{equation}
\label{qexp}
q_{i} = -(h_{i} + \sum_{j \neq i} (1-c_{j}^{(i)})h_{j} - c_{0}^{(i)}).
\end{equation}
In the proof of claim \eqref{pform2} for $(\deg_{i}p_{i},\deg_{k}p_{k})=(1,0)$ we had $\deg_{k}p_{i}=0$ and thus $c_{k}^{(i)}=0$.

 For $(\deg_{i}p_{i},\deg_{k}p_{k})=(1,1)$, we have 
\[p_{k} = (h_{k} + \sum_{j \neq k} c_{j}^{(k)}h_{j} + c_{0}^{(k)})\]
and equation \eqref{eq4} in this case reads 
\[c_{i}^{(k)}(h_{i} + \sum_{j \neq i} c_{j}^{(i)}h_{j} + c_{0}^{(i)}) = c_{k}^{(i)}(h_{k} + \sum_{j \neq k} c_{j}^{(k)}h_{j} + c_{0}^{(k)}).\]
Applying $\mathtt{c}_{i}$ and $\mathtt{c}_{k}$ separately to this equation we get
\[\begin{cases}
   c_{i}^{(k)}=c_{k}^{(i)}c_{i}^{(k)},\\
   c_{k}^{(i)}=c_{i}^{(k)}c_{k}^{(i)};
  \end{cases}
\]
which has solutions $(c_{k}^{(i)},c_{i}^{(k)}) \in \{(0,0),(1,1)\}$. Thus, in particular, $c_{k}^{(i)} \in \{0,1\}$.

Next, suppose $(\deg_{i}p_{i},\deg_{k}p_{k})=(1,2)$. Then $q_{k}=-1$ and, by Lemma~\ref{lemma3e}, $q_{i}$ does not depend on $h_{k}$. But then from our explicit form \eqref{qexp}
 of $q_{i}$ above we see that $1-c_{k}^{(i)}=0$.
We have now proved that any $p_{i}$ with $\deg_{i}p_{i}=1$ has the prescribed form.

Next, let $\deg_{i}p_{i}=2$. Then 
\[p_{i}=(h_{i} + \sum_{j \neq i} c_{j}^{(i)}h_{j} + c_{0}^{(i)})(h_{i} + \sum_{j \neq i} (1-c_{j}^{(i)})h_{j} - c_{0}^{(i)}-1)\]
and $q_{i}=-1$. By assumption, $\deg_{k}p_{k} \neq 0$. For $\deg_{k}p_{k} = 1$, the polynomial 
$p_{k}$ satisfies \eqref{pform3} and divides $p_{i}$ by Lemma~\ref{lemmatech}. Since the coefficient at
$h_{k}$ in $p_{k}$ is $1$, we obtain either $c_{k}^{(i)} = 1$ or $1-c_{k}^{(i)}=1$, as desired.
Finally, we consider the case
 $(\deg_{i}p_{i},\deg_{k}p_{k})=(2,2)$.
We shall again show that $c_{k}^{(i)} \in \{0,1\}$. Let
\[p_{i}=\alpha_{1}\alpha_{2}  , \quad p_{k}=\beta_{1}\beta_{2} \quad \text{ and } q_{i}=q_{k}=-1,\]
where
\begin{equation}
 \label{eq7}
\begin{cases}
    \alpha_{1}=h_{i} + \sum_{j \neq i} c_{j}^{(i)}h_{j} + c_{0}^{(i)},\\
    \alpha_{2}=h_{i} + \sum_{j \neq i} (1-c_{j}^{(i)})h_{j} - c_{0}^{(i)}-1,\\
    \beta_{1}=h_{k} + \sum_{j \neq k} c_{j}^{(k)}h_{j} + c_{0}^{(k)},\\
    \beta_{2}=h_{k} + \sum_{j \neq k} (1-c_{j}^{(k)})h_{j} - c_{0}^{(k)}-1.
  \end{cases}
\end{equation}
Then the equality 
$p_{i}\sigma_{i}p_{k} = p_{k}\sigma_{k}p_{i}$ 
from Lemma~\ref{lemma3e} becomes
\begin{equation}
\label{eq6}
\alpha_{1}\alpha_{2}(\beta_{1}-c_{i}^{(k)})(\beta_{2}-(1-c_{i}^{(k)}))=\beta_{1}\beta_{2}(\alpha_{1}-c_{k}^{(i)})(\alpha_{2}-(1-c_{k}^{(i)})).
\end{equation}
If $\alpha_{1}\alpha_{2} | \beta_{1}\beta_{2}$, formula \eqref{eq6} becomes $\sigma_{i}(p_{i}) = \sigma_{k}(p_{i})$ which implies
\[p_{i} \in \bb{C}[h_{i}+h_{k},h_{1},\ldots, h_{i-1},h_{i+1}, \ldots, h_{k-1},h_{k+1}, \ldots, h_{n}] \simeq (\ca{P}_{i} \cap \ca{P}_{k})[h_{i}+h_{k}].\]
Here we view $(\ca{P}_{i} \cap \ca{P}_{k})[h_{i}+h_{k}]$ as the subring of $\ca{P}$ consisting of polynomials in the single variable $(h_{i}+h_{k})$ with coefficients in $\ca{P}_{i} \cap \ca{P}_{k}$.
This contradicts the fact that $\sigma_{i}^{-1}(p_{i})q_{i} = -h_{i}(\olh+1) + \tilde{g}_{i}$ from Lemma~\ref{lemma3e}. Thus,
 without loss of generality, $\alpha_{1}$ does not divide $\beta_{1}\beta_{2}$. Moreover, $\alpha_{1}$ only divides $(\alpha_{1}-c_{k}^{(i)})$ if $c_{k}^{(i)}=0$, so we
 may assume that $\alpha_{1}$ divides $(\alpha_{2}-(1-c_{k}^{(i)}))$. Considering the coefficient at $h_{i}$, this happens only if  $\alpha_{1}=\alpha_{2}-(1-c_{k}^{(i)})$, so formula \eqref{eq6} becomes
\[\alpha_{2}(\beta_{1}-c_{i}^{(k)})(\beta_{2}-(1-c_{i}^{(k)}))=\beta_{1}\beta_{2}(\alpha_{2}-1).\]
This shows that $\alpha_{2}$ divides $\beta_{1}$ or $\beta_{2}$. Without loss of generality we may assume that $\alpha_{2} = c\beta_{1} \text{ for some }c\in\bb{C}^{*}$.
 Comparing the coefficient of $h_{k}$ in our explicit expressions of $\alpha_{2}$ and $\beta_{1}$ in 
 formula \eqref{eq7}, we get $c=(1-c_{k}^{(i)}) \neq 0$ and thus
\begin{equation}
 \label{eq8}
(\beta_{1}-c_{i}^{(k)})(\beta_{2}-(1-c_{i}^{(k)}))=\beta_{2}(\beta_{1}-(1-c_{k}^{(i)})^{-1}).
\end{equation}
The four prime factors occurring in \eqref{eq8} all have $h_{k}$-coefficient $1$, so one divides another if and only if they are equal.
Now $\beta_{2} = (\beta_{2}-(1-c_{i}^{(k)}))$ gives $c_{i}^{(k)}=1$ and thus we have $(\beta_{1}-1)=(\beta_{1}-(1-c_{k}^{(i)})^{-1})$, which gives $c_{k}^{(i)}=0$.
 On the other hand, $\beta_{2}=(\beta_{1}-c_{i}^{(k)})$ gives
$(\beta_{1}-1)=(\beta_{1}-(1-c_{k}^{(i)})^{-1})$, which also implies $c_{k}^{(i)}=0$. We have now considered all cases, and can conclude that $c_{k}^{(i)}\in \{0,1\}$ always.
\end{proof}

From here on, we shall assume that $\deg_{i}p_{i} \in \{1,2\}$ for all $i \in \oln$. By Proposition~\ref{symmprop}, all other module structures can then be obtained by application of the functor $\mathrm{F}_{\tau}$.
Define a binary relation $\sim$ on $\oln$ by 
\begin{equation}
 \label{eq9}
i \sim j \Longleftrightarrow \text{$p_{i}$ and $p_{j}$ share a common prime divisor.}
\end{equation}
Note that $\sim$ is symmetric and reflexive (since we assume that all $\deg_{i}p_{i} \geq 1$), but it is in general not transitive: for example we would have $h_{1} \sim h_{1}h_{2} \sim h_{2}$
 while $h_{1}\not\sim h_{2}$.
 
\begin{prop}
\label{pqprop}
If $\deg_{k} p_{k}=1$, then for some $b_{k} \in \mathbb{C}$ we have
\[p_{k}=\big( \sum_{j\sim k}h_{j}+b_{k} \big)\qquad \text{ and }\qquad  q_{k}=-\big( h_{k} + \sum_{j\not\sim k}h_{j} -b_{k} \big).\]
\end{prop}
\begin{proof}
Since $\deg_{k} p_{k}=1$ by Lemma~\ref{lemmatech}, the polynomial $p_{k}$ divides all polynomials $p_{j}$ of degree $2$ and hence the coefficient at $h_{j}$ in $p_{k}$ is $1$ for all $j$ with $\deg_{j} p_{j}=2.$
Assume now instead that $p_{j}$ has degree $1$ as well. Then equation \eqref{eq4} becomes 
$c_{k}^{(j)}p_{k} = c_{j}^{(k)}p_{j}$. Now if $j \sim k$, then $p_{j}=p_{k}$ and  $c_{j}^{(k)}=1$. On the other hand, if $j \not\sim k$, then $p_{j} \neq p_{k}$ which implies $c_{j}^{(k)}=0$,
 so the formula for $p_{k}$ in the proposition is correct and $q_{k}$ is uniquely determined by $p_{k}$ (see formula \eqref{qexp} in the proof of Lemma~\ref{lemmapform}).
\end{proof}

Next we prove that either all the first degree polynomials $p_{i}$ coincide, or they are pairwise different.
\begin{lemma}
\label{lemmadeg12}
Let \[\oln_{1} := \{i \in \oln |\deg_{i} p_{i}=1\}\; \; \text{ and let } \; \; \oln_{2} := \{j \in \oln |\deg_{j} p_{j}=2\} = \oln \setminus \oln_{1}.\]
Then either $p_{i}=p_{j}$ for all $i,j \in \oln_{1}$ or $p_{i}\neq p_{j}$ for all distinct $i,j \in \oln_{1}$.
\end{lemma}
\begin{proof}
 For each $k\in \oln$ define \[A_{k}:=\{t \in \oln | t \sim k\}.\]
Suppose there exists indices $i,k \in \oln_{1}$ with $p_{i} \neq p_{k}$.
The statement of Proposition~\ref{pqprop} can now be written as follows:
\begin{equation}
 \label{eq10}
q_{k}=-\big( h_{k} + \sum_{t \in  \oln \setminus A_{k}}h_{t} -b_{k} \big).
\end{equation}
  But now, since $i \not\sim k$, we have $i \in \oln \setminus A_{k}$ and $k \in \oln \setminus A_{i}$, so equality \eqref{eq5} becomes just $q_{i}=q_{k}$. 
Using our explicit expressions for $q_{i}$ and $q_{k}$ from \eqref{eq10} we see that this is
 equivalent to $b_{i}=b_{k}$ and $\{i\} \cup (\oln \setminus A_{i}) = \{k\} \cup (\oln \setminus A_{k})$,
which simplifies to
\begin{equation}
\label{eq11}
 A_{i} \setminus \{i\} = A_{k} \setminus \{k\}.
\end{equation}
Now for $j\in \oln_{2}$ we always have $i \sim j$ and for $j \in \oln_{1}$ we have $i\sim j$ only if $p_{i}=p_{j}$. 
Thus $A_{i} \cap \oln_{1}$ and $A_{k} \cap \oln_{1}$ are disjoint, so \eqref{eq11} implies $A_{i}=\oln_{2} \cup \{i\}$ and $A_{k}=\oln_{2} \cup \{k\}$.
 In particular, $p_{k} \neq p_{j}$ for any $j \in \oln_{1}$ with $j\neq k$.
 This shows that all the polynomials $p_{j}$ with $j\in \oln_{1}$ are pairwise distinct.
 \end{proof}

\subsection{Classifications of objects in $\ca{M}$}
We are now ready to classify objects in $\ca{M}$ for $\mathfrak{sl}_{n+1}$.
\begin{defi}
\label{structure}
Let $S \subset \oln$ and $b \in \bb{C}$. Define $M_{b}^{S}$ to be the set $\ca{P}$ equipped with the following $\g$-action:

\begin{displaymath}
\begin{array}{rcl}
h_{k}\cdot f&:=&h_{k}f, \qquad  k\in \oln;\\
\\
e_{i,n+1}\cdot f&:=&\begin{cases}
                     (\olh + b)\sigma_{i}f, & i \in S,\\
		     (\olh + b)(h_{i}-b-1)\sigma_{i}f, & i\not\in S;
                    \end{cases}
 \\
\\
e_{n+1,j}\cdot f&:=&\begin{cases}
                     -(h_{j}-b) \sigma_{j}^{-1}f, & j \in S,\\
		     -\sigma_{j}^{-1}f,  & j\not\in S;
                    \end{cases}
 \\
\\
e_{i,j}\cdot f&:=&\begin{cases}
                     (h_{j}-b)    			 \sigma_{i}\sigma_{j}^{-1}f, & i,j \in S,\\
							 \sigma_{i}\sigma_{j}^{-1}f, & i \in S, j\not\in S,\\
                     (h_{i}-b-1)(h_{j}-b)    		 \sigma_{i}\sigma_{j}^{-1}f, & i \not\in S,j \in S,\\
		     (h_{i}-b-1)    			 \sigma_{i}\sigma_{j}^{-1}f, & i,j\not\in S.\\
                    \end{cases}
\end{array}
\end{displaymath}
To write this more compactly we introduce the indicator functions $\delta_{P}$ where $P$ is some statement, and
$\delta_{P}=1$ if $P$ is true and $\delta_{P}=0$ if $P$ is false. Then the above can be written as follows.
 \[ \begin{cases}
 h_{k} \cdot f = h_{k}f ,&   \\ 
 e_{i,n+1} \cdot f = (\olh + b)(\delta_{i\in S} + \delta_{i \not\in S}(h_{i}-b-1))\sigma_{i}f, &   \\ 
 e_{n+1,j} \cdot f =-(\delta_{j\in S}(h_{j}-b) + \delta_{j \not\in S})\sigma_{j}^{-1}f, &   \\ 
 e_{i,j} \cdot f =(\delta_{i\in S} + \delta_{i \not\in S}(h_{i}-b-1))(\delta_{j\in S}(h_{j}-b) + \delta_{j \not\in S})     \sigma_{i}\sigma_{j}^{-1}f. &   \\ 
\end{cases}
\]

\end{defi}
\begin{thm}
Equipped with the action of Definition~\ref{structure}, $M_{b}^{S}$  is a $\g$-module for all $b\in \mathbb{C}$ and all $S \subset \oln$.
\end{thm}
\begin{proof}
First note that for all $k \in \oln$ and all $i,j \in {\bf n+1}$ with $i \neq j$ we have $e_{i,j}\cdot h_{k}=h_{k}-\delta_{i,k}+\delta_{j,k}$.
But then for all $f \in M_{b}^{S}$ we have
\begin{align*}
  e_{i,j} \cdot h_{k} \cdot f - h_{k} \cdot e_{i,j} \cdot  f &= e_{i,j} \cdot (h_{k}  f) - h_{k} ( e_{i,j} \cdot  f) \\
&= (h_{k}-\delta_{i,k}+\delta_{j,k})(e_{i,j} \cdot  f) - h_{k} ( e_{i,j} \cdot  f)\\
&= (-\delta_{i,k}+\delta_{j,k})e_{i,j} \cdot  f \\
&= [e_{i,j},h_{k}] \cdot f
\end{align*}
where we used Lemma~\ref{lemma1} in the last step. Thus the relation
\[[e_{i,j},h_{k}] \cdot f = e_{i,j} \cdot h_{k} \cdot f - h_{k} \cdot e_{i,j} \cdot  f\]
holds for all $k \in \oln$ and all $i,j \in {\bf n+1}$ with $i \neq j$.

The remaining relations are more time consuming to check. We first check that for all $i,j,k \in \oln$  (with $i \neq j$) and all $f \in M_{b}^{S}$ we have
\begin{equation}
 \label{eq12}
e_{n+1,k} \cdot e_{i,j} \cdot f - e_{i,j} \cdot  e_{n+1,k} \cdot f = [e_{n+1,k},e_{i,j}] \cdot f.
\end{equation}
Expressing the left side of \eqref{eq12} explicitly we get
\begin{align*}
e_{n+1,k} &\cdot e_{i,j} \cdot f - e_{i,j} \cdot  e_{n+1,k} \cdot f\\
=&-(\delta_{k\in S}(h_{k}-b)+\delta_{k\not\in S}) \sigma_{k}^{-1} \big( (\delta_{i \in S} + \delta_{i \not\in S}(h_{i}-b-1))(\delta_{j \in S}(h_{j}-b)+\delta_{j \not\in S})  \sigma_{i}\sigma_{j}^{-1}f \big)\\
&+ (\delta_{i \in S} + \delta_{i \not\in S}(h_{i}-b-1))(\delta_{j \in S}(h_{j}-b)+\delta_{j \not\in S}) \sigma_{i}\sigma_{j}^{-1} \big( (\delta_{k\in S}(h_{k}-b)+\delta_{k\not\in S}) \sigma_{k}^{-1} f\big) \\
=&-(\delta_{k\in S}(h_{k}-b)+\delta_{k\not\in S})(\delta_{i \in S} + \delta_{i \not\in S}(h_{i}-b-1+\delta_{i,k}))(\delta_{j \in S}(h_{j}-b + \delta_{j,k})+\delta_{j \not\in S})\sigma_{k}^{-1}\sigma_{i}\sigma_{j}^{-1}f \\
&+ (\delta_{i \in S} + \delta_{i \not\in S}(h_{i}-b-1))(\delta_{j \in S}(h_{j}-b)+\delta_{j \not\in S})(\delta_{k\in S}(h_{k}-b-\delta_{i,k}+\delta_{j,k})+\delta_{k\not\in S}) \sigma_{i}\sigma_{j}^{-1}\sigma_{k}^{-1} f.\\
\end{align*}
Now, when expanding this expression, all terms not containing $\delta_{i,k}$ or $\delta_{j,k}$ will cancel by symmetry. Thus, by factoring out $\delta_{i,k}$, $\delta_{j,k}$ and $\delta_{i,k}\delta_{j,k}$ separately,
 we can rewrite this as
\begin{align*}
e_{n+1,k} &\cdot e_{i,j} \cdot f - e_{i,j} \cdot  e_{n+1,k} \cdot f\\
=& \delta_{i,k} \Big[ -(\delta_{k\in S}(h_{k}-b)+\delta_{k\not\in S})(\delta_{i \not\in S})(\delta_{j \in S}(h_{j}-b)+\delta_{j \not\in S}) \\
 & \qquad +(\delta_{i \in S} + \delta_{i \not\in S}(h_{i}-b-1))(\delta_{j \in S}(h_{j}-b)+\delta_{j \not\in S})(-\delta_{k\in S}) \Big]\sigma_{i}\sigma_{j}^{-1}\sigma_{k}^{-1} f \\
 &+\delta_{j,k} \Big[ -(\delta_{k\in S}(h_{k}-b)+\delta_{k\not\in S})(\delta_{i \in S} + \delta_{i \not\in S}(h_{i}-b-1))(\delta_{j\in S})\\
 & \qquad +(\delta_{i \in S} + \delta_{i \not\in S}(h_{i}-b-1))(\delta_{j \in S}(h_{j}-b)+\delta_{j \not\in S})(\delta_{k\in S})   \Big]\sigma_{i}\sigma_{j}^{-1}\sigma_{k}^{-1} f \\
 &+\delta_{i,k}\delta_{j,k} \Big[ -(\delta_{k\in S}(h_{k}-b)+\delta_{k\not\in S})(\delta_{i \not\in S} \delta_{j \not\in S})  \Big]\sigma_{i}\sigma_{j}^{-1}\sigma_{k}^{-1} f \\
=& \delta_{i,k} \Big[ -(\delta_{i\in S}(h_{i}-b)+\delta_{i\not\in S})(\delta_{i \not\in S})(\delta_{j \in S}(h_{j}-b)+\delta_{j \not\in S}) \\
 & \qquad +(\delta_{i \in S} + \delta_{i \not\in S}(h_{i}-b-1))(\delta_{j \in S}(h_{j}-b)+\delta_{j \not\in S})(-\delta_{i\in S}) \Big]\sigma_{j}^{-1} f \\
 &+\delta_{j,k} \Big[ -(\delta_{j\in S}(h_{j}-b)+\delta_{j\not\in S})(\delta_{i \in S} + \delta_{i \not\in S}(h_{i}-b-1))(\delta_{j\in S})\\
 & \qquad +(\delta_{i \in S} + \delta_{i \not\in S}(h_{i}-b-1))(\delta_{j \in S}(h_{j}-b)+\delta_{j \not\in S})(\delta_{j\in S})   \Big]\sigma_{i}\sigma_{j}^{-2} f \\
 &+\delta_{i,k}\delta_{j,k} \Big[ -(\delta_{k\in S}(h_{k}-b)+\delta_{k\not\in S})(\delta_{i \not\in S} \delta_{j \in S})  \Big]\sigma_{j}^{-1} f. \\
\end{align*}
Now, since $i \neq j$, we have $\delta_{i,k}\delta_{j,k}=0$ so the last term is zero.  Using the fact that $\delta_{P \land Q}=\delta_{P}\delta_{Q}$ and that $\delta_{\lnot P}=1-\delta_{P}$,
 the above expression can be further simplified to
\begin{align*}
e_{n+1,k} &\cdot e_{i,j} \cdot f - e_{i,j} \cdot  e_{n+1,k} \cdot f\\
=& \delta_{i,k} \Big[ -(\delta_{i \not\in S})(\delta_{j \in S}(h_{j}-b)+\delta_{j \not\in S}) \\
 & \qquad -(\delta_{j \in S}(h_{j}-b)+\delta_{j \not\in S})(\delta_{i\in S}) \Big]\sigma_{j}^{-1} f \\
 &+\delta_{j,k} \Big[ -(h_{j}-b)(\delta_{i \in S} + \delta_{i \not\in S}(h_{i}-b-1))(\delta_{j\in S})\\
 & \qquad +(\delta_{i \in S} + \delta_{i \not\in S}(h_{i}-b-1))(h_{j}-b)(\delta_{j\in S})   \Big]\sigma_{i}\sigma_{j}^{-2} f \\
=&\delta_{i,k} \Big[  -(\delta_{i\in S} + \delta_{i \not\in S})(\delta_{j \in S}(h_{j}-b)+\delta_{j \not\in S}) \Big]\sigma_{j}^{-1} f\\
=&-\delta_{i,k}   (\delta_{j \in S}(h_{j}-b)+\delta_{j \not\in S}) \sigma_{j}^{-1} f \\
=&\delta_{i,k} e_{n+1,j} \cdot f \\
=&[e_{n+1,k},e_{i,j}] \cdot f. 
\end{align*}
Thus \eqref{eq12} holds for all $i,j,k \in \oln$ (with $i \neq j$) and all  $f \in M_{b}^{S}$ as required.

The remaining relations lead to similar equations which are left to the reader to verify.
\end{proof}

\begin{thm}
\label{class}
For $n>1$, let $\ca{S}$ be the full subcategory of $\ca{M}$ consisting of all modules of form
\[\mathrm{F}_{a}(M_{b}^{S})\qquad \text{ and }\qquad  \mathrm{F}_{a}\circ \mathrm{F}_{\tau}(M_{b}^{S})\]
for all $a=(a_{1}, a_{2}, \ldots, a_{n}, 1) \in (\bb{C}^{*})^{n} \times \{1\}$, $S \subset \oln$ and $b\in \bb{C}$. Then $\ca{S}$ is a skeleton in  $\ca{M}$.
\end{thm}
\begin{proof}
Suppose we are given a module structure on $\ca{P}$ determined by $(p_{1}, \ldots, p_{n},q_{1}, \ldots, q_{n})$ in accordence with Proposition~\ref{prop1}.
First assume that $\deg_{k} p_{k} \in \{1,2\}$ for all $k \in \oln$. By Lemma~\ref{lemmalc} and Remark~\ref{lead1rmk}, it suffices
 to prove that $\mathtt{c}_{k}(p_{k})=1$ for all $k\in \oln$ implies that the module is isomorphic to
 either $\mathrm{F}_{a}(M_{b}^{S})$ or  $\mathrm{F}_{a}\circ \mathrm{F}_{\tau}(M_{b}^{S})$ for suitable choices of $a \in (\bb{C}^{*})^{n+1}$, $b\in \bb{C}$ and $S \subset \oln$. 
Thus we assume that the leading coefficient of each $p_{k}$ is $1$ and $\deg_{k}p_{k} \in \{1,2\}$.
Note that for $n=1$ we obtain the same set of modules as in Section~\ref{sl2}.
 Define 
\[N:=\#\{i \in \oln \; | \; \deg_{i} p_{i}=2\}.\]
We consider first the case $N=0$ when all $p_{k}$ have $k$-degree $1$.
By Proposition~\ref{pqprop}, each $p_{k}$ and thus the module structure is completely determined by the relation $\sim$ from \eqref{eq9}, which in this case becomes an equivalence relation:
\[i \sim k \Longleftrightarrow p_{i}=p_{k}.\]

By Lemma~\ref{lemmadeg12}, either all $p_{i}$ are pairwise distinct or they all coincide.
Using the explicit expression of $p_{i}$ from Proposition~\ref{pqprop}, in the first case (i.e. all $p_{i}$ are distinct) we have
 \[p_{i}=h_{i} + b_{i} \text{ for all $i \in \oln$},\] which implies that $q_{i}=\olh-b_{i}$ for all $i$ and by Remark~\ref{divcor} all the scalars $b_{i}$ coincide (write $b:=b_{i}$) and the module is isomorphic to
$\mathrm{F}_{(1,1,\ldots, 1,-1)} \circ \mathrm{F}_{\tau}(M_{b}^{\oln})$. In the other case (i.e. all $p_{i}$ coincide) we get
 \[p_{i}=\olh + b \text{ for all $i \in \oln$},\] which makes the module isomorphic to $M_{b}^{\oln}$. 

Next, suppose $N=1$ and let $\deg_{k}p_{k}=2$. For $n=2$, the unique $p_{i}$ of degree $1$ divides both $p_{i}$ and $p_{k}$, so the module is isomorphic to $M_{b}^{\{i\}}$ for some $b \in \bb{C}$.
For $n=3$, write $\{1,2,3\}=\{i,j,k\}$. If $p_{i}=p_{j}$, then the module is isomorphic to $M_{b}^{\oln \setminus \{k\}}$ for some $b \in \bb{C}$. If $p_{i} \neq p_{j}$, since both divide $p_{k}$,
 we get $p_{k}=p_{i}p_{j}$ and $q_{k}=-1$. This implies that $p_{i}=h_{i}+h_{k}+b_{i}$ and $p_{j}=h_{j}+h_{k}+b_{j}$ and, by Lemma~\ref{lemma3e}, we get first $b_{i}+b_{j}+1=0$ and then $b_{i}=b_{j}=-\frac{1}{2}$.
 But this would determine the module structure completely and one can check that we for example would have $e_{2,4} \cdot e_{1,3} \cdot 1 - e_{1,3} \cdot e_{2,4} \cdot 1 \neq [e_{2,4},e_{1,3}] \cdot 1$, so
 this module structure is not possible.

 Finally, for $n>3$ we have at least three degree $1$ polynomials
 dividing $p_{k}$ and the latter polynomial has two prime factors. Thus at least two of the divisors coincide and hence they all coincide by Lemma~\ref{lemmadeg12}. This gives a module isomorphic to  $M_{b}^{\oln \setminus \{k\}}$
 for some $b\in \bb{C}$. Thus the statement of the theorem holds for $N=1$.

We now turn to the case $2 \leq N < n$. Let $i,k \in \oln$ be distinct indices such that $\deg_{i}p_{i} = 2 = \deg_{k}p_{k}$. Then $p_{i}$ and $p_{k}$ share a common prime factor $\alpha$
 by Lemma~\ref{lemmatech}. So we have $p_{i}=\alpha \beta_{i}$ and $p_{k} = \alpha \beta_{k}$ where $\alpha, \beta_{i}, \beta_{k}$ are pairwise distinct. 
But then all $p_{j}$ for $j \in \oln_{1}$ share a factor with both $p_{i}$ and $p_{k}$, so $p_{j}= \alpha$ for all $j \in \oln_{1}$. Thus $\alpha | p_{j}$ for all $j \in \oln$ and
 our module is isomorphic to $M_{b}^{\oln_{1}}$ for some $b \in \bb{C}$. 

Finally, we consider the case $N=n$ where all polynomials $p_{i}$ have degree $2$. For $n=2$, 
the polynomials $p_{1}$ and $p_{2}$ share a common factor which then divides all $p_{i}$ for $i \in \{1,2\}$ and the
 module is isomorphic to $M_{b}^{\varnothing}$ for some $b \in \bb{C}$. For $n \geq 3$, suppose that not all polynomials $p_{i}$ share the same factor. Then there exist distinct indices $i,j,k$ such that
\[p_{i} = \alpha \beta, \; p_{j} = \alpha \gamma\quad  \text{ and }\quad  p_{k} = \beta \gamma\]
for some pairwise distinct $\alpha,\beta,\gamma$. But then a fourth $p_{j}$ cannot share a common factor with both $p_{i},p_{j},$ and $p_{k}$ so for $n>3$ all $p_{i}$ share a fixed factor $\alpha$ for all $i \in \oln$,
 and the module is isomorphic to $M_{b}^{\varnothing}$ for some $b \in \bb{C}$. Thus the only remaining case is
$n=3$ and
\[p_{1} = \alpha \beta, \; p_{2} = \alpha \gamma\quad \text{ and }\quad  p_{3} = \beta \gamma.\]
But then we explicitly have
\[\alpha=h_{1}+h_{2}+b_{1}, \beta=h_{1}+h_{3}+b_{2}\quad  \text{ and }\quad  \gamma = h_{2} + h_{3} + b_{3}.\] Moreover, 
$q_{1}=q_{2}=q_{3}=-1$, so Lemma~\ref{lemma3e} implies that $b_{1}=b_{2}=b_{3}=-\frac{1}{2}$. But this does not give a module structure, since, for example,
 $e_{2,4} \cdot e_{1,3} \cdot 1 - e_{1,3} \cdot e_{2,4} \cdot 1 \neq [e_{2,4},e_{1,3}] \cdot 1$.

We have now proved the theorem in case $\deg_{k}p_{k} \in \{1,2\}$ for all $k \in \oln$. Suppose now this is not the case for some module $M$. By Lemma~\ref{lemmatech}, we then have
 $\deg_{k}p_{k} \in \{0,1\}$ for all $k \in \oln$ and $\deg_{k}q_{k} \in \{1,2\}$ for all $k \in \oln$. Thus we apply the theorem above to $\mathrm{F}_{\tau}(M)$ instead, giving
 either $\mathrm{F}_{\tau}(M) \simeq \mathrm{F}_{a}(M_{b}^{S})$ or $\mathrm{F}_{\tau}(M) \simeq \mathrm{F}_{\tau} \circ \mathrm{F}_{a}(M_{b}^{S})$ for some $a,b,S$.
 Applying $F_{\tau}$ again we get either $M \simeq \mathrm{F}_{a^{-1}} \circ \mathrm{F}_{\tau}(M_{b}^{S})$ or $M \simeq \mathrm{F}_{a^{-1}}(M_{b}^{S})$.

Thus we have proved that every module in $\ca{M}$ is isomorphic to one of the representatives in the theorem.
 It remains to show that different module structures are not isomorphic.

Any morphism $\varphi$ in $\ca{M}$ is determined by its value at $1$ since $\varphi(f)=f \varphi(1)$.
Suppose now that $\varphi: M \rightarrow M'$ is an isomorphism, where we identify $M$ and $M'$ with $\ca{P}$ as $\uh$-modules. Then \[1=\varphi(\varphi^{-1}(1))=\varphi^{-1}(1)\varphi(1),\]
which shows that $\varphi(1)\in \bb{C}^{*}$ and $\varphi=c{\bf 1}$ for some $c\in \bb{C}^{*}$. For all $i \in \oln$ we define $p_{i}:=e_{i,n+1} \cdot 1 \in M$
 and $p_{i}':=e_{i,n+1} \cdot 1 \in M'$. Then we have
\[cp_{i}'=p_{i}'\sigma_{i} \varphi(1) = e_{i,n+1} \cdot \varphi(1) = \varphi(e_{i,n+1} \cdot 1) = \varphi(p_{i}) = p_{i} \varphi(1)=cp_{i},\]
which gives $p_{i}=p_{i}'$. Thus two modules can be isomorphic only if 
for every $i \in \oln$ the action of the element $e_{i,n+1}$  on these two modules is given by the same
polynomial. Similarly, the polynomials $q_{j}$ must coincide in isomorphic modules. 
Now in a module $\mathrm{F}_{a}(M_{b}^{S})$, the set of $p_{i}$'s are uniquely determined by $a,b,S$ so there are no nontrivial isomorphisms between objects of this form, and hence the same also holds for
 objects of form $\mathrm{F}_{a}  \circ F_{\tau} (M_{b}^{S})$. Finally, for the module $\mathrm{F}_{a}(M_{b}^{S})$, since $n > 1$ each $p_{i}$ has an irreducible component $(\olh + b)$ which $q_{i}$ does not have.
 This shows that there are no isomorphisms $\mathrm{F}_{a}  \circ F_{\tau} (M_{b}^{S}) \rightarrow \mathrm{F}_{a'} (M_{b'}^{S'})$.
\end{proof}

\begin{rmk}
 Note that the theorem applies also to $n=1$, except that in this case we have isomorphisms $M_{b}^{\bf 1} \simeq \mathrm{F}_{(1,-1)} \circ \mathrm{F}_{\tau}(M_{b}^{\bf 1})$ for each $b \in \bb{C}$.
 Via the relations $M_{b} \simeq M_{b}^{\bf 1}$ and $M_{b}' \simeq \mathrm{F}_{\tau}(M_{b}^{\varnothing})$, we recover again the results of Theorem~\ref{sl2classi}.
\end{rmk}

\subsection{Simples and subquotients}
\label{simplicity}
We now show that the modules of $\ca{M}$ generically are simple.
We start with some sufficient conditions for simplicity.
\begin{thm}
\label{simples}
\begin{enumerate}[$($i$)$]
 \item\label{simples1}  For $b \in \bb{C}$ with $(n+1)b \not\in \bb{N}_{0}$, the module $M_{b}^{\oln}$ is simple.
 \item\label{simples2}  For $S \neq \oln$, the module $M_{b}^{S}$ is simple.
\end{enumerate}
\end{thm}
\begin{proof}
Fix a subset $S \subset \oln$. We start by constructing a new basis of $\uh$. For each $i \in \oln$ and for all integers $k \geq -1$, we define \[H_{k}^{(i)} := \prod_{j=0}^{k}(h_{i}-b+j).\]
Then the set \[B:=\{H_{k_{1},k_{2}, \ldots, k_{n}}:=H_{k_{1}}^{(1)}\cdots H_{k_{n}}^{(n)} |  k_{1}, \ldots , k_{n} \geq -1\}\]
is a basis for $\uh$. For each $i \in \oln$ we also define
\[ A_{i}:= \begin{cases}
            e_{n+1,i} + (h_{i}-b)  & i \in S\\
	    (h_{i}-b)(e_{n+1,i}+1) & i \not\in S
           \end{cases}.\]
Then for each $i \in \oln$ we have
\[A_{i}\cdot H_{k}^{(i)} = -(k+1)H_{k}^{(i)},\]
 and since $A_{i}$ commutes with $H_{k}^{(j)}$ for all $j \neq i$ we deduce that \[A_{i} \cdot H_{k_{1},k_{2}, \ldots, k_{n}} = -(k_{i}+1) H_{k_{1},k_{2}, \ldots, k_{n}},\]
so $H_{k_{1},k_{2}, \ldots, k_{n}}$ is an eigenvector of $A_{i}$ with eigenvalue $-(k_{i}+1)$. This shows that the elements $A_{1}, \ldots, A_{n}$ act
 diagonally in the basis $B$, where each generalized eigenspace is of dimension $1$.  
 We conclude that any submodule of $M_{b}^{S}$ is the span of some subset of elements from $B$.

 Let $V$ be a nonzero submodule of $M_{b}^{S}$, and let $f$ be a nonzero element of $V$.

Now let $H_{k_{1},k_{2}, \ldots, k_{n}}$ be a basis element occurring with nonzero coefficient in $f$ expressed in the basis $B$. Then  $H_{k_{1},k_{2}, \ldots, k_{n}} \in V$.
 We shall show that for each $i \in \oln$, if $k_{i} \neq -1$, we have also $H_{k_{1}, \ldots, k_{i}-1, \ldots, k_{n}} \in V$;
 it will then follow by induction that $1=H_{-1, \ldots, -1} \in V$ which implies $V=M_{b}^{S}$.

 Fix $i \in \oln$. If $i \not\in S$ we note that 
 \[(e_{n+1,i}+1) \cdot H_{k_{1},k_{2}, \ldots, k_{n}} = -(k_{i}+1)\sigma_{i}^{-1}(H_{k_{1}, \ldots, k_{i}-1, \ldots,  k_{n}}),\]
 so for $k_{i} \geq 0$, by considering the $i$-degree, we see that when we express this in the basis $B$ the coefficient of $H_{k_{1}, \ldots, k_{i}-1, \ldots,  k_{n}}$ is nonzero which implies that
 $H_{k_{1}, \ldots, k_{i}-1, \ldots,  k_{n}} \in V$ as required.

On the other hand, if $i \in S$, we act by $e_{i,n+1}$ on $H_{k_{1},k_{2}, \ldots, k_{n}}$ and express the result in our basis $B$:
\begin{align*}
  e_{i,n+1} &\cdot H_{k_{1},k_{2}, \ldots, k_{n}} =(\olh +b)( H_{k_{1}, \ldots,  k_{n}}  -(k_{i}+1) H_{k_{1}, \ldots, k_{i}-1, \ldots,  k_{n}})\\
=& \olh H_{k_{1}, \ldots,  k_{n}} - (k_{i}+1)\olh H_{k_{1}, \ldots, k_{i}-1, \ldots,  k_{n}}\\
&+b ( H_{k_{1}, \ldots,  k_{n}}  -(k_{i}+1) H_{k_{1}, \ldots, k_{i}-1, \ldots,  k_{n}})  \\
=& (\sum_{j=1}^{n} (h_{j}-b+k_{j}+1)) H_{k_{1}, \ldots,  k_{n}} + (nb - \sum_{j=1}^{n}(k_{j}+1))H_{k_{1}, \ldots,  k_{n}} \\
&- (k_{i}+1)((h_{i}-b+k_{i}) + \sum_{j \neq i} (h_{j}-b+k_{j}+1)) H_{k_{1}, \ldots, k_{i}-1, \ldots,  k_{n}}\\
&\qquad -(k_{i}+1) (nb -k_{i} - \sum_{j \neq i} (k_{j}+1)) H_{k_{1}, \ldots, k_{i}-1, \ldots,  k_{n}}\\
&+b ( H_{k_{1}, \ldots,  k_{n}}  -(k_{i}+1) H_{k_{1}, \ldots, k_{i}-1, \ldots,  k_{n}})  \\
=& (\sum_{j=1}^{n}  H_{k_{1}, \ldots, k_{j}+1, \ldots,  k_{n}}) + (nb - \sum_{j=1}^{n}(k_{j}+1))H_{k_{1}, \ldots,  k_{n}} \\
&- (k_{i}+1)(H_{k_{1}, \ldots,  k_{n}} + \sum_{j \neq i}  H_{k_{1},  \ldots, k_{j}+1, \ldots,  k_{i}-1, \ldots,  k_{n}})\\
&\qquad -(k_{i}+1) (nb -k_{i} - \sum_{j \neq i} (k_{j}+1)) H_{k_{1}, \ldots, k_{i}-1, \ldots,  k_{n}}\\
&+b ( H_{k_{1}, \ldots,  k_{n}}  -(k_{i}+1) H_{k_{1}, \ldots, k_{i}-1, \ldots,  k_{n}})  
\end{align*}
Thus we see that the coefficient of $H_{k_{1},  \ldots,  k_{i}-1, \ldots,  k_{n}}$ in   $e_{i,n+1} \cdot H_{k_{1},k_{2}, \ldots, k_{n}}$
 is precisely 
\[-(k_{i}+1)((n+1)b-k_{i}-\sum_{j \neq i}(k_{j}+1)).\]
Now if $(n+1)b$ is not a natural number, this quantity is nonzero since $k_{i}+\sum_{j \neq i}(k_{j}+1) \in \bb{N}_{0}$ while $(n+1)b \not\in \bb{N}_{0}$.
 This proves the induction step and implies the simplicity of $M_{b}^{S}$ for $(n+1)b \not\in \bb{N}_{0}$, which in particular proves \eqref{simples1}.

 To prove \eqref{simples2}, assume that $S \neq \oln$.
 Suppose again that $H_{k_{1}, \ldots, k_{n}}\in V$ with $k_{i} \neq -1$ for some $i$. We observe from the calculation above that the coefficient of
 $H_{k_{1}, \ldots, k_{i}-1, \ldots ,k_{j}+1, \ldots, k_{n}}$ in $e_{i,n+1} \cdot H_{k_{1}, \ldots,  k_{n}}$
 is nonzero for each $j \neq i$, thus $H_{k_{1}, \ldots, k_{i}-1, \ldots ,k_{j}+1, \ldots, k_{n}}$ belongs to $V$ also. Acting repeatedly by $e_{i,n+1}$ for all $i \in S$ we obtain
 finally $H_{k_{1}', \ldots,  k_{n}'} \in V$ where $k_{j}' = -1$ for all $j \in S$. Acting repeatedly by $(e_{n+1,k}+1)$ for all $k \not\in S$ we finally obtain $1 \in V$ so $V$ is simple.
\end{proof}

Since the functors $\mathrm{F}_{a}$ and $\mathrm{F}_{\tau}$ from Section~\ref{functors} are equivalences we also have the following corollary.
\begin{cor}
 For $(n+1)b \not\in \bb{N}_{0}$ or $S \neq \oln$, the modules $\mathrm{F}_{a}(M_{b}^{S})$ and $\mathrm{F}_{a} \circ \mathrm{F}_{\tau}(M_{b}^{S})$ are simple.
\end{cor}
It turns out that any simple module in $\ca{M}$ are of the form in the above corollary.
The only case remaining is when both $S= \oln$ and $(n+1)b$ is a natural number, this is covered in the following theorem.
\begin{thm}
For $(n+1)b \in \bb{N}_{0}$, the module $M_{b}^{\oln}$ has a unique proper submodule $W$ which is simple and belongs to $\mathfrak{M}$ but not to $\ca{M}$.
 The corresponding simple quotient has dimension ${(n+1)b +n \choose n}$ and has lowest weight $\lambda$,
 where $\lambda(h_{i}) = b -\delta_{i,1}(n+1)b$.
\end{thm}
\begin{proof}
Define \[W:=span\{H_{k_{1}, \ldots, k_{n}} | \sum_{i=1}^{n} k_{i} \geq (n+1)b-(n-1)\}.\]
This is a submodule of $M_{b}^{\oln}$. From the calculation in the proof of Theorem~\ref{simples} it is clear that any vector in $M_{b}^{\oln} \slash W$ can be reduced to $1$ so the module
 $M_{b}^{\oln} \slash W$ is simple. Similarly one shows that $W$ is simple.
 Define \[v:= \prod_{k=0}^{(n+1)b-1}(h_{1}-b+k) = H_{(n+1)b-1,-1,-1, \ldots, -1} \in M_{b}^{\oln} \slash W.\]
 Now $(h_{1}-b+(n+1)b) \cdot v = H_{(n+1)b,-1,-1, \ldots, -1}=0$ in the quotient. Similarly, for $i>1$ we have $(h_{i}-b) \cdot v =0$, so $v$ is a weight vector.
From Definition~\ref{structure} we see that $e_{n+1,j} \cdot v = -(h_{j}-b)\sigma_{j}^{-1} v = 0$ for all $j \in \oln$. For $1 \leq j < i \leq n$ we also
 obtain $e_{i,j} \cdot v = (h_{j}-b)\sigma_{i}\sigma_{j}^{-1} v = 0$, since $i \neq 1$. This means that $v$ is killed by $\mathfrak{n_{-}}$.

Thus we have showed that $v$ is a lowest weight vector in $M_{b}^{\oln} \slash W$ of weight $\lambda$ where $\lambda(h_{i}) = b -\delta_{i,1}(n+1)b$ as stated in the theorem.
The dimension of $M_{b}^{\oln} \slash W$ equals the number of ways to choose $n$ integers $k_{i} \geq-1$ such that their sum is less than $(n+1)b-(n-1)$.
Thus we obtain \[\dim M_{b}^{\oln} \slash W = { (n+1)b +n \choose n}.\]
\end{proof}

\noindent Department of Mathematics, Uppsala University, Box 480, SE-751 06, Uppsala, Sweden, email: jonathan.nilsson@math.uu.se

\end{document}